 \documentclass[draft]{article}

\usepackage{amsmath,amsfonts,amsthm,amssymb,amscd,cancel,color,mathabx}
\usepackage{enumitem}
\usepackage{verbatim}
\usepackage[dvipdfmx]{graphicx}
\renewcommand{\Re}{\mathop{\rm Re}\nolimits}
\renewcommand{\Im}{\mathop{\rm Im}\nolimits}
\def\S{\mathhexbox278}

\theoremstyle{plain}
\newtheorem{theorem}{Theorem}[section]
\newtheorem{lemma}[theorem]{Lemma}
\newtheorem{proposition}[theorem]{Proposition}

\theoremstyle{definition}

\theoremstyle{remark}
\newtheorem{remark}[theorem]{Remark}
\newtheorem{notation}[theorem]{Notation}

\newtheorem{assumption}[theorem]{Assumption}

\newcommand{\R}{{\mathbb R}}

\newcommand{\Z}{{\mathbb Z}}

\newcommand{\N}{{\mathbb N}}

\def\im{{\rm i}}

\newcommand{\C}{\mathbb{C}}

\def\({\left(}
\def\){\right)}
\def\<{\left\langle}
\def\>{\right\rangle}
\newcommand{\sech}{{\mathrm{sech}}}

\newcommand{\rad}{{\mathrm{rad}}}    \newcommand{\diag}{{\mathrm{diag}}}

\numberwithin{equation}{section}

\begin{document}

\title{On the asymptotic stability   of  ground states of the pure power  NLS on the line at 3rd and 4th order Fermi Golden Rule
\normalsize \\
  }

\author{Scipio Cuccagna, Masaya Maeda }
\maketitle


\begin{abstract}Assuming   as hypotheses the results proved numerically by Chang et al. \cite{Chang} for the exponent $p\in (3,5)$,  we prove that some of  the ground states of the nonlinear Schr\"odinger equation (NLS) with pure power nonlinearity of exponent $p$ in the line are asymptotically stable for a certain set of values of the exponent $p$ where the FGR occurs by means of a discrete mode 3rd or 4th order power interaction with the continuous mode.  For the  3rd  the result is true for generic $p$  while for the 4th order case we assume that there are $p$'s satisfying Fermi Golden rule and the non-resonance condition of the threshold of the continuous spectrum of the linearization.
The argument is similar to our recent result valid for $p$ near 3 contained in \cite{CM24D1}.
\end{abstract}

\section{Introduction}

We consider  the pure    power focusing  Nonlinear Schr\"odinger Equation  (NLS) on the line
\begin{align}\label{eq:nls1}&
  \im \partial _t  u +\partial _x^2  u = -f(u)  \text{   where }    f(u)=|u| ^{p-1}u    \text{  with $p>1$. }
\end{align}
 We  consider only even solutions, eliminating translations  and simplifying the problem. In particular, we will study Equation \eqref{eq:nls1} in the space $H^ 1_\rad (\R )= H^ 1_\rad (\R ,\C ) $, of even functions in $H^ 1  (\R ,\C )$.
 As it is well known, the pure power NLS has a scaling property:
   if $u(t,x)$ solves Equation \eqref{eq:nls1}, then for any $\omega>0$ so does $\omega^{\frac{1}{p-1}}u(\omega t, \sqrt{\omega}x)$.
 Well known   is   the  existence of standing waves,  of  the form $u(t,x)=e^{\im t}\phi(x)$, with the explicit formula
 \begin{align} &
     \phi (x)= \phi_p (x) =  {\(\frac {p+1}2 \)^{\frac 1{p-1}}}{
 \sech   ^{\frac 2{p-1}}\(\frac{p-1}2 x\)} ,\label{eq:sol}
    \end{align}
    see formula (3.1) of Chang et al. \cite{Chang},
    and by scaling, setting $\phi _\omega (x)=\omega ^{\frac 1{p-1}} \phi (\sqrt{\omega }x) $,  standing waves $u(t,x)=e^{\im \omega t}\phi_{\omega}(x)$.

The variational characterization of $\phi_{\omega}$ is given by using the invariants of Equation \eqref{eq:nls1}.
For  Energy $\mathbf{E}$ and Mass $\mathbf{Q}$ given by
  \begin{align}\label{eq:energy}
& \mathbf{E}( {u})=\frac{1}{2}   \| u' \| ^2 _{L^2\( \R \) } -\int_{\R} F(u)    \,dx   \text {  where }  F(u)    =\dfrac{ |u|  ^{p+1}}{p+1}, \\&   \label{eq:mass} \mathbf{Q}( {u})=\frac{1}{2}   \| u  \| ^2 _{L^2\( \R \) },
\end{align}
it is known that
   $\phi _\omega$ is a minimizer of $\mathbf{E}$ under the constraint $\mathbf{Q}=\mathbf{Q}(\phi _\omega )=:\mathbf{q}(\omega )$. Notice that $\mathbf{q}(\omega ) =\omega ^{ \frac{2}{p-1}-\frac{1}{2}}   \mathbf{q}(1)$.
     We have $\nabla \mathbf{E} (\phi _\omega)= -\omega   \nabla \mathbf{Q} (\phi _\omega)  $ which reads also
\begin{align} \label{eq:static}
   - \phi _\omega ''+\omega \phi _\omega - f(\phi _\omega)=0   .
\end{align}
Consider now  for $\omega ,\delta \in \R_+:=(0,\infty)$ the set     $$\mathcal{U} (\omega  ,\delta  ) := \bigcup _{\vartheta _0\in \R }   e^{\im \vartheta _0}  D_{H ^1_\rad (\R )}({\phi}_{\omega },\delta  ), $$ where $D_X(u,r):=\{v\in X\ |\ \|u-v\|_X<r \}$.
The following was shown by Cazenave and Lions \cite{cazli}, see also Shatah  \cite{shatah} and Weinstein  \cite{W1}.

\begin{theorem}[Orbital Stability]  Let $p\in (1,5)$ and let $\omega _0 >0$. Then for any $\epsilon >0$  there exists a  $\delta >0$ such that for any  initial value    $u_0\in \mathcal{U} (\omega _0,\delta  ) $,   the corresponding solution  satisfies
$u\in C^0  \(  \R ,   \mathcal{U} (\omega _0,\epsilon  ) \) $.

\end{theorem}
\qed

In order to study the notion of asymptotic stability, like in finite dimension, it is useful to have information on the  \textit{linearization} of \eqref{eq:nls1} at $\phi _\omega$, which   has the following form,
\begin{align}\label{eq:lineariz2} \partial _t   \begin{pmatrix}
 r_1 \\ r_2
 \end{pmatrix}      &=  \mathcal{L}_{\omega }  \begin{pmatrix}
 r_1 \\ r_2
 \end{pmatrix}  \text{  with }    \mathcal{L}_\omega := \begin{pmatrix}
0 & L_{-\omega} \\ -L_{+\omega} & 0
\end{pmatrix} ,
    \end{align}
where    \begin{align}\label{eq:lin1}&
  L _{+\omega}:=- \partial _x^2   +\omega -p \phi _\omega ^{p-1},   \\& \label{eq:lin0} L _{-\omega}:=- \partial _x^2   +\omega-  \phi _\omega ^{p-1}.
\end{align}
The linearization is better seen in the context of functions  in $H^ 1_\rad (\R , \R ^2 ) $ rather than in $H^ 1_\rad (\R ,\C ) $, because it is $\R$--linear rather than $\C$--linear.
  We have
\begin{align}
  \label{scaling}  S_{\omega}u(x):=\omega^{1/4}u(\sqrt{\omega}x) \Longrightarrow   \mathcal{L}_{\omega} = \omega S_{\omega} \mathcal{L} S_{\omega}^{-1} \text{ for } \mathcal{L}=\mathcal{L}_1.
\end{align}
Thus, for the spectral properties of $L_{\pm\omega}$ and $\mathcal{L}_{\omega}$, it suffices to consider the case $\omega=1$.
\begin{proposition}\label{prop:knownspec}
Let $\mathcal{ L}=\mathcal{ L}_1$.
Then we have the following.
\begin{enumerate}
\item[(1)] Spectrum of $\mathcal{ L}$ is symmetric with respect to real and imaginary axis.
\item[(2)] Essential spectrum of $\mathcal{ L}$ is $(-\im \infty,-\im]\cup[\im ,\im \infty)$.
\item[(3)] $\mathcal{ L}$ has no eigenvalues $\im \lambda $ with $\lambda > 1$ or $\lambda<-1$.
That is, $\mathcal{L}$ has no embedded eigenvalues.
\item[(4)] For $p\in (2,5)\setminus\{3\}$, $\mathcal{ L}$ has exactly one eigenvalue of the form $\im \lambda(p)$  with $0<\lambda(p)<1$.
Further, we have $\lim_{p\to 3+}\lambda(p)=1$, $\lim_{p\to 5-}\lambda(p)=0$ and $\lambda(4)>1/2$.
\item[(5)] $\dim \(  \mathcal{L} -\im \lambda (p)   \) =1$, where $\lambda(p)$ is given in (4).
\end{enumerate}
\end{proposition}

\begin{proof}
For (1), see  Krieger and Schlag \cite{KrSch}.
(2) follows from the exponential decay of $\phi^{p-1}$.
The proof of (3) is in  \cite[p.p.914-915]{KrSch}, following an argument  in Perelman \cite{perelman}.
(4) The existence of eigenvalue $\im \lambda$ satisfying $0<\lambda<1$ follows from Theorem 3.7 of Chang et al.\ \cite{Chang} and the uniqueness follows from Theorem 3.8 of  \cite{Chang}.
For $\lim_{p\to 5-}\lambda(p)=0$, see Corollary 2.8 of \cite{Chang} and for $\lim_{p\to 3+}\lambda(p)=1$, see Coles and Gustafson \cite{coles}.
The estimate $\lambda(4)>\sqrt{3}/2>1/2$ is given in Theorem 3.8 of \cite{Chang}.
For $|p-3|\ll1$, (5) is proved in \cite{coles} and for $p\in (3,5)$ it follows from continuity.
\end{proof}

We will prove later the following, which is essentially  already known.
\begin{proposition}[No Threshold Resonance]
   \label{prop:nonres}   Consider the following statement,
\begin{equation}\label{eq:nores111}
  \text{ there is   no nonzero bounded  solution of $\mathcal{L} u=\im   u$.}
\end{equation}
Consider the set
\begin{equation}\label{eq:defhatF}
  \widehat{{\mathbf{F}}} :=\{p \in ( 1, +\infty ) \ |\  \text{statement \eqref{eq:nores111} is true }\} .
\end{equation}
Then  $( 1,  +\infty ) \backslash   \widehat{{\mathbf{F}}}$ is a discrete subset of $( 1, +\infty )$.
\end{proposition}

  Coles and Gustafson \cite{coles} prove that   \eqref{eq:nores111} is true for  $0<|p-3|\ll 1$. We will sketch   here  an alternative proof in Section \ref{sec:prop:nonres}  based on  Krieger and Schlag \cite{KrSch} in the context of the proof of  Proposition  \ref {prop:nonres}.

    Chang et al. \cite{Chang} show numerically that  the function $(3,5)\ni p\to \lambda (p) $   is  strictly decreasing.
In this paper, we assume the following.
\begin{assumption}\label{ass:eigenvalues1} We assume $\lambda'(p)<0$ for $p\in (3,5)$.
%
  \end{assumption}

\begin{remark}
\noindent  The operator $\mathcal{L}$ is analytic of type $(A)$ in $p\in (1,5)$. This implies that the function
$(3,5)\to \lambda (p)$ is analytic and so in particular is differentiable, see \cite[Ch. 12]{reedsimon}.
\end{remark}

  By (4) of Proposition \ref{prop:knownspec} and Assumption \ref {ass:eigenvalues1},
  \begin{align} &\nonumber
   \text{there are numbers  $4<p_2<p_3<p_4<5$ such that  $\lambda (p_n )= 1/n  $ for $n=2,3,4$  and} \\& \text{
$1/(n-1)>\lambda (p)>1/ n $ for $p_{n-1}<p<p _n$.} \label{eq:p1p2}
  \end{align}
  We call 3rd order Fermi Golden Rule (FGR)      the equations when    $1/2>\lambda (p)>1/3 $, that is for $p\in (p_2,p_3)$, and   4th order   FGR      the equations when    $1/3>\lambda (p)>1/4 $, that is for $p\in (p_3,p_4)$.

%


We will need a  FGR assumption, whose significance will be clear  below in Section \ref{sec:fgr}.
%
%
%
%
%
%
%
Let $n=3$ or $4$.
Let $\gamma_n(p)$ (defined for $p\in (p_{n-1},p_n)$) be the constant introduced in \eqref{eq:fgrgamma} and set
\begin{equation}\label{eq:deftildeF}
  \widetilde{F}_n:=\{p\in (p_{n-1},p_n)\ |\ \gamma_n(p)\neq 0\}.
\end{equation}

For the case $n=3$ we have the following result.
\begin{proposition}\label{prop:n3FGR}
Assume Assumption \ref{ass:eigenvalues1}.
Then,
$(p_2,p_3)\setminus \widetilde{F}_3$ is a discrete subset of    $(p_2,p_3)$.
\end{proposition}

We do not have an analogue of Proposition \ref{prop:n3FGR} for  $n=4$.
%
%
 We set now
\begin{equation}\label{eq:deftildeF}
  {F}_n:=  \widetilde{F}_n  \cap \(  ( 1,  +\infty ) \backslash   \widehat{{\mathbf{F}}}\) .
\end{equation}
%


\noindent
We set  $\xi[\omega] \in H^1(\R, \C^2) $ be  an appropriately normalized generator of $\ker (\mathcal{L} -\im \lambda (p))$, see Section \ref{sec:lin1}.

\begin{theorem} \label{thm:asstab} Let here $n=3$ or $4$.  Assume     Assumption  \ref{ass:eigenvalues1}.
Then
  for any $p\in F _n $
  and $\epsilon >0$   there exists a  $\delta >0$ such that for any  initial value    $u_0\in  D _{H^1_\rad (\R ) } ({\phi},\delta  ) $
there exist functions $ \vartheta \in C^1 \( \R , \R \)$, $\omega\in C^1(\R,(0,\infty))$ and  $ z \in C^1 \( \R , \C  \)$  and constants  $\omega _\pm >0$ such that    the  solution  of \eqref{eq:nls1} with initial datum  $u_0$  can be written as
 \begin{align} \label{eq:asstab1}
    & u(t)= e^{\im \vartheta (t)} \(   \phi _{\omega (t)} + z(t) \xi [\omega(t) ] + \overline{z}(t) \overline{\xi}[\omega(t) ] +\eta (t)\) \text{   with}
 \\&  \label{eq:asstab2}   \int _{\R } \|  e^{- \< x\>}   \eta (t ) \| _{H^1(\R )}^2 dt <  \epsilon  \text{  where }\< x\>:=\sqrt{1+x^2}, \\&  \label{eq:asstab20}
   \lim _{t\to \infty  }  \|  e^{- \< x\>}   \eta (t ) \| _{L^2(\R )}     =0 ,  \\&  \label{eq:asstab2}
   \lim _{t\to \infty}z(t)=0  \text{ and }  \\& \label{eq:asstab3}
   \lim _{t\to \pm  \infty}\omega (t)= \omega _\pm   .
\end{align}

\end{theorem}

\begin{remark}
If $F_4=\emptyset$, then the statement of Theorem 1.7 is trivial because the assumption cannot be satisfied.
We conjecture that as the case $n=3$, $(p_3,p_4)\setminus F_4$ is discrete.
\end{remark}

\begin{remark}
Fix $\epsilon>0$ arbitrary and take $\delta>0$ to be the constant given in Theorem \ref{thm:asstab}.
Suppose that $u_0$ satisfies
\begin{align}\label{eq:u0scale}
    \|u_0-\phi_{\omega_0}\|_{H^1}< \min(\omega_0,\omega_0^{-1})^{1/4}\omega_0^{1/(p-1)} \delta.
\end{align}
Then, setting $v(t,x)=\omega_0^{-1/(p-1)}u(\omega_0^{-1}t,\omega_0^{-1/2}x)$, we have
\begin{align*}
\|v(0,x)-\phi\|_{H^1}^2=&\|\omega_0^{-1/(p-1)} (u(0,\omega_0^{-1/2}x)-\phi_{\omega_0}(\omega_0^{-1/2}x))\|_{L^2}^2\\&+\|\omega_0^{-1/(p-1)-\frac{1}{2}} ((\partial_x u)(0,\omega_0^{-1/2}x)-\phi_{\omega_0}'(\omega_0^{-1/2}x))\|_{L^2}^2\\
=& \omega_0^{\frac{1}{2}-\frac{2}{p-1}}\| u(0)-\phi_{\omega_0}\|_{L^2}^2+\omega_0^{-\frac{1}{2}-\frac{2}{p-2}}\|\partial_x u(0)-\phi_{\omega_0}'\|_{L^2}^2\\
\leq &\max(\omega_0,\omega_0^{-1})^{1/2} \omega_0^{-\frac{2}{p-2}}\|u(0)-\phi_{\omega_0}\|_{H^1}^2\\
<&\delta^2.
\end{align*}
Since $v$ is also solution of \eqref{eq:nls1}, we can apply Theorem \ref{thm:asstab} to $v$.
Rescaling back, we get analogue result for solutions satisfying \eqref{eq:u0scale}.
\end{remark}

Equation \eqref{eq:nls1} is a  classical Hamiltonian systems in PDE's and the asymptotic stability of its ground states has been a longstanding open problem. For an extensive introduction   and for many other references  we refer to our recent paper \cite{CM24D1}  which deals with a different interval of the exponents $p$. Here we only reiterate that,  as  we wrote in   \cite{CM24D1}, our results  here are possible thanks to the idea initiated in  Kowalczyk et al. \cite{KMM20}, further developed and    refined in   \cite{KMM3,KMMvdB21AnnPDE,KM22}, to use virial estimates to prove dispersion and in particular to their astute and  simple way to treat the nonlinear term by an integration by parts,  see for example  \cite[p.2142]{KMM3}. In \cite{CM24D1} we discussed at length our choice to preserve the first virial inequality of   Kowalczyk et al. but to replace the second virial inequality with smoothing estimates.
In fact, while a result similar to \cite{CM24D1} was proved by Rialland \cite{Rialland2} following closely  the  theory in Martel \cite{Martel2},    Theorem \ref{thm:asstab} seems to be beyond  the scope of  the theory developed by Martel \cite{Martel1,Martel2} and Rialland \cite{rialland,Rialland2}, for two reasons. The first is that  Martel \cite{Martel2} and Rialland \cite{Rialland2} consider only 2nd order interactions, like \cite{CM24D1}. The second reason is that their theory, and in particular the 2nd virial inequality,  requires their systems to be small perturbations of the cubic NLS. Here we emphasize    that our framework is   unaffected by the  order of the FGR, except for  the construction of the Refined Profile, which involves a simple Taylor expansion argument. We recall that the Refined Profile is an elementary alternative to the much more complicated normal forms argument in papers such as \cite{Cu1}, since the modulation coordinates obtained using the Refined Profile in the ansatz \eqref{eq:ansatz} do not require further coordinate transformations.
The only reason we do not treat  the   FGR at order     5 or  higher for equation \eqref{eq:nls1}   is that $u\to f(u)$ is not regular enough for us to write the Taylor expansion of $f(u)$ needed for the Refined Profile.
Our approach has some technical constraints, since we need to exclude eigenvalues embedded in the continuous spectrum of the linearization. While, as we wrote above,  thanks to   Krieger and Schlag \cite{KrSch} and Perelman \cite{perelman}  here  we know that there are no embedded eigenvalues, in general this issue is delicate. Notice though that generically embedded eigenvalues are expected not to exist (we conjecture, but we have no particular evidence for this conjecture,   that for linearizations at ground states such embedded eigenvalues never exist).

 We emphasize that in  \cite{CM24D1}, like in  Martel \cite{Martel2} and in some of the cases  in Rialland \cite{Rialland2}    there are no assumptions and that everything is proved. In particular in  \cite{CM24D1} there is a proof, inspired by the analogous proof in Martel \cite{Martel2}, of the Fermi Golden Rule.  More precisely \cite{CM24D1} has a sketch of the proof of the FGR, but in  \cite{CM242} we write all the computations on the FGR omitted in \cite{CM24D1}.

Given Assumption \ref{ass:eigenvalues1},
the   proof of Theorem \ref{thm:asstab} is similar to the corresponding proof in \cite{CM24D1}. There are only two differences. The first is that the Refined Profile here is different from  the one in \cite{CM24D1}. The other   is that here we simplify  the proof of the  weighted estimate contained in Proposition \ref{lem:smooth111}, which here  we generalize considerably.  In \cite{CM24D1} the analogous result was contingent on a condition on the Wronskian between two Jost functions of the linearization. Here we give a much more general and simple  proof which drops the discussion on  the Wronskian and easily extends  to the linearized operators under the condition \ref{eq:nores111} and which do not have embedded eigenvalues.

\section{Linearization}\label{sec:lin1}

We return to a discussion of the linearization \eqref{eq:lineariz2} like in \cite{CM24D1}.
Weinstein \cite{W2}  showed that  for $ 1<p<5$ the   generalized kernel $ {N}_g(\mathcal{L}_\omega):=\cup_{j=1}^\infty \mathrm{ker} \mathcal{L}_{\omega}^j$ in $H^1 _\rad ( \R , \C^2 ) $ is
\begin{align}\label{eq:Ng}
 {N}_g(\mathcal{L}_{\omega })=\mathrm{span}\left \{ \begin{pmatrix}
 0 \\ \phi_{ \omega}
 \end{pmatrix} , \begin{pmatrix}
 \partial_\omega \phi_{\omega } \\ 0
 \end{pmatrix}          \right \} .
\end{align}
By symmetry reasons,   it is known that the spectrum $\sigma \(  \mathcal{L}_{\omega }\) \subseteq \C $ is symmetric by reflection with respect to the coordinate axes. Furthermore,
by Krieger and Schlag   \cite[p. 909]{KrSch}  we know that  $\sigma \(  \mathcal{L}_{\omega }\) \subseteq \im \R  $. By standard Analytic  Fredholm theory   the essential spectrum is $(- \infty \im ,-\omega \im ] \cup
[ \im \omega , +\infty \im ) $.

\noindent Let us consider the direct sum  decomposition
\begin{align}
  \label{eq:dirsum} L^2_\rad (\R , \C ^2) =  {N}_g(\mathcal{L}_{\omega })\bigoplus {N}_g^\perp (\mathcal{L}_{\omega }^*).
\end{align}  We have,  for $\lambda  = \omega \lambda (p)$,   a further decomposition
\begin{align}
  \label{eq:dirsum1}&{N}_g^\perp (\mathcal{L}_{\omega }^*) =   \ker (\mathcal{L}_{\omega }-\im \lambda )\bigoplus \ker (\mathcal{L}_{\omega }+\im \lambda ) \bigoplus X_c (\omega ) \text{ where } \\& X_c (\omega ) : = \(   {N}_g(\mathcal{L}_{\omega }^*) \bigoplus  \ker (\mathcal{L}_{\omega }^*-\im \lambda )\bigoplus \ker (\mathcal{L}_{\omega }^*+\im \lambda ) \) ^{\perp} . \label{eq:dirsum11}
\end{align}
We denote by $P_c$ the projection of $ L^2_\rad (\R , \C ^2)$  onto $X_c (\omega )$ associated with the above decompositions.

\noindent The space $L^2_\rad(\R , \C ^2)$ and the action of $\mathcal{L}_{\omega }$ on it are obtained by first  identifying $L^2_\rad(\R , \C  ) =L^2_\rad( \R ,  \R ^2  )  $ and then by extending   this action to  the completion of
$ L^2_\rad (\R ,  \R ^2  ) \bigotimes _\R \C  $ which is identified with $L^2_\rad (\R , \C ^2)$.  In $\C$ we consider the inner product
\begin{align*}
   \< z, w \> _{\C} =\Re \{z\overline{w}    \} =z_1w_1+z_2w_2  \text{ where } a_1=\Re a, \  a_2=\Im a \text{ for }a=z,w.
\end{align*}
This obviously coincides with the inner product in $\R ^2$ and expands as the standard inner product  $  \< X  , Y  \> _{\C ^2}  =  X ^ \intercal \overline{Y}$  (row column product, vectors here are columns)  form   in $\C ^2$.  The operator of multiplication by $\im $  in $\C=\R^2$ extends into the linear operator
$J ^{-1}=-J$ where   \begin{align}
J=\begin{pmatrix}
0 & 1 \\ -1 & 0
\end{pmatrix}.\nonumber
\end{align}
For  $u,v\in L^2_\rad (\R , \C ^2)$ we set $ \< u , v  \>    :=\int _\R \< u (x), v (x) \> _{\C ^2}  dx$. We have a natural symplectic form given by $\Omega :=\<  J ^{-1}\cdot  , \cdot   \>$ in both $ L^2(\R , \C ^2)$ and
$ L^2_\rad (\R , \R ^2)=L^2_\rad (\R , \C  )$, where equation \eqref{eq:nls1} is the Hamiltonian system  in  $  L^2_\rad (\R , \C  )$     with Hamiltonian the energy $E$ in \eqref{eq:energy}.  As we mentioned we consider
a generator $\xi[\omega]\in  \ker (\mathcal{L}_{\omega }-\im  \lambda )$.  Then for the complex conjugate $\overline{\xi}[\omega] \in  \ker (\mathcal{L}_{\omega }+\im \lambda )$.
Notice   the well known and elementary $J \mathcal{L}_{\omega }=-\mathcal{L}_{\omega } ^*J$  implies that  $ \ker (\mathcal{L}_{\omega }^*+\im \lambda )= \mathrm{span}\left \{  J \xi[\omega] \right \}$ and $ \ker (\mathcal{L}_{\omega }^*-\im \lambda )= \mathrm{span}\left \{  J \overline{{\xi}[\omega]} \right \}$. Notice that in   Lemma 2.7  \cite{CM2} it is shown that we can normalize $\xi[\omega]$  so that, as done in \cite{CM24D1} where it is explained why the sign in the right is a minus,
\begin{align}
  \label{eq:xinormlize} \Omega (   {\xi}[\omega],  {\xi}[\omega] ) =- \im .
\end{align}
Furthermore, notice that
\eqref{eq:xinormlize} is the same as
\begin{align}
  \label{eq:xinormlize2} \Omega (   \Re {\xi}[\omega] ,  \Im {\xi}[\omega] ) = \frac{1}{2}  \text{  and }   \Omega (   \Re {\xi}[\omega] ,   \Re {\xi}[\omega] )=\Omega (    \Im  {\xi}[\omega] ,   \Im  {\xi}[\omega] )=0,
\end{align}
where the latter is immediate by the skewadjointness of $J$. Like in \cite{CM24D1}   we can normalize so that
\begin{align}\label{eq:reimxi1} &  \xi [\omega] = \(   \xi _1 , \xi _2 \) ^\intercal  \text{  with }  \xi _1=\Re  \xi _1 \text{  and }  \xi _2=\im \Im  \xi _2.
\end{align}
Hence condition \eqref{eq:xinormlize} becomes
\begin{align}
  \label{eq:xinormliz1} \int _{\R} \xi _1 \Im \xi _2 dx = 2 ^{-1}.
\end{align}

\begin{notation}\label{not:notation} We will use the following miscellanea of  notations and definitions.
\begin{enumerate}



\item Like in the theory of     Kowalczyk et al.   \cite{KMM3},      we   consider constants  $A, 
 \epsilon , \delta  >0$ satisfying \begin{align}\label{eq:relABg}\log(\delta ^{-1})\gg\log(\epsilon ^{-1}) \gg     A
   \gg 1. \end{align}
\begin{remark}
In \cite{KMM3} there was another constant $B$ satisfying $A^{1/3}\gg B\gg 1$.
In this paper, we will use $B=A^{1/3}$.
\end{remark}

\item For  $\kappa \in (0,1)$    fixed in terms of $p$  and  small enough, we consider
\begin{align}\label{eq:l2w}&
\|  {\eta} \|_{  L ^{p,s}} :=\left \|\< x \> ^s \eta \right \|_{L^p(\R )}   \text{  where $\< x \>  := \sqrt{1+x^2}$,}\\&
\| \eta  \|_{  { \Sigma }_A} :=\left \| \sech \(\frac{2}{A} x\) \eta '\right \|_{L^2(\R )} +A^{-1}\left \|    \sech \(\frac{2}{A} x\)  \eta   \right\|_{L^2(\R )}  \text{ and} \label{eq:normA}\\& \|  {\eta} \|_{ \widetilde{\Sigma} } :=\left \| \sech \( \kappa  x\)   {\eta}\right \|_{L^2(\R )} . \label{eq:normk}
\end{align}
\item   For $\gamma\in \R$ and $n\in \N$ we set
\begin{align}
L^p_\gamma&:=\{u\in \mathcal D'(\R,\C)\ |\ \|u\|_{L^p_\gamma}:=\|e^{\gamma\<x\>}u\|_{L^p}<\infty\}  \text { for $p\in [1,+\infty ]$},\label{L2g}\\
H^n_\gamma&:=\{u\in \mathcal D'(\R,\C)\ |\ \|u\|_{H^n_\gamma}:=\|e^{\gamma \<x\>}u\|_{H^n}<\infty\}.\label{H1g}
\end{align}

\item We will consider the Pauli matrices \begin{equation*}
   \sigma_1=\begin{pmatrix} 0 &
   1 \\
   1 & 0
    \end{pmatrix} \,,
   \quad
   \sigma_2= \begin{pmatrix} 0 &
   -\im \\
    \im  & 0
    \end{pmatrix} \,,
   \quad
   \sigma_3=\begin{pmatrix} 1 &
   0 \\
   0 & -1
    \end{pmatrix}.
   \end{equation*}

\item We have denoted by $P_c(\omega)$ the projection on the space \eqref{eq:dirsum11} associated to the spectral decomposition
\eqref{eq:dirsum1}  of the operator $\mathcal{L} _{\omega}$. Later in \eqref{eq:opH1} we introduce an operator $H$ which is equivalent to $\mathcal{L}$. By an abuse of notation we will use the same symbol $P_c$ to denote the   continuous spectrum projections of both $H$ and $\mathcal{L}$.

\item We denote by $C^ \omega (\Omega , X)$ the space of analytic functions from an open subspace $\Omega $ of $\R$ or $\C$ with values in a Banach space $X$.

 \end{enumerate}

\end{notation} \qed

The following results, for $p\in \widehat{{\mathbf{F}}}\cap (1,5)$, are stated at this point also in \cite{CM24D1}.  By
 the absence of embedded eigenvalues $\mathcal{L}_{\omega }$ is admissible in the sense of Krieger and Schlag \cite{KrSch} and so the following is true.
\begin{proposition}\label{prop:KrSch} For any $\omega >0$ and $s>3/2$   there is a  constant $C>0$ 
such that
\begin{align} \label{eq:KrSch}
  \| P_c(\omega) e^{t \mathcal{L}_{\omega }} : L ^{2,s}  (\R , \C ^2) \to  L ^{2,-s} (\R , \C ^2)  \| \le C_\omega \< t \> ^{-\frac{3}{2}} \text{  for all $t\in \R$.}
\end{align}

\end{proposition}\qed

The following is proved \cite{CM24D1} utilizing the fact that $0<|p-3|\ll 1$ and properties of the operator $\mathcal{L}_{\omega }$ for $p=3$. Here we will give a simpler proof that uses only the no resonance condition.
\begin{proposition} \label{lem:smoothest1} For any $\omega>0$,  $s>3/2$ and $\tau >1/2$ there exists a constant $C>0 $ such that
 \begin{align}&   \label{eq:smoothest11}   \left \|   \int   _{0} ^{t   }e^{  (t-t')  \mathcal{L}_{\omega }}P_c(\omega )g(t') dt' \right \| _{L^2( \R ,L^{2,-s} (\R ))  } \le C  \|  g \| _{L^2( \R , L^{2,\tau} (\R ) ) } \text{ for all $g\in  L^2( \R , L^{2,\tau} (\R ) )$}.
\end{align}
\end{proposition}

The proof of the analogous result in  \cite{CM24D1} to the following one,   based only on  the absence of threshold resonance, which for $0<|p-3|\ll 1$ follows by Coles and Gustafson \cite{coles}, applies also in the current paper and will be not further discussed.

\begin{proposition}[Kato smoothing]\label{lem:smooth111} For any $\omega >0$ and $s>1$  there exists $C>0$   such that
\begin{equation}\label{eq:smooth111}
   \|   e^{t \mathcal{L}_\omega}P_c(\omega) u_0 \| _{L^2(\R , L ^{2,-s}(\R   ))} \le C \|  u_0 \| _{L^2(\R  )}.
\end{equation}
\end{proposition}

\section{Refined profile }\label{sec:refprof}

It is well known,  see Weinstein \cite{W2},  that the following is    a symplectic  submanifold  of  $L ^2_\rad (\R, \C ) $,
\begin{align}
  \label{eq:mangs} \mathcal{S}=\left \{  e^{\im \vartheta  } {\phi}_{\omega} : \vartheta \in \R , \omega >0     \right \} .
\end{align}
  Now we will introduce the notion of Refined Profile, which here is more complicated than in \cite{CM24D1}. We will follow the framework  introduced in more general form in \cite{CM3}.  Here however things are much simpler.
For $\mathbf{m}\in \N_0^{2 }$ where $\N_0=\{  0,1,2,...     \}$ of the form $\mathbf{m}=(m_1,m_2)$ we set $\overline{\mathbf{m}}=(m_2,m_1)$,  \begin{align*}
 {z}^{\mathbf{m}}= {z}^{ {m}_1}\overline{ {z}}^{ {m}_2}  \text{   and }|\mathbf{m}|=m_1+ m_2.
\end{align*}
We will treat separately the construction of the Refined Profile   for the 3rd and the  4th order FGR.
Before doing this, we clarify the dependence on $\omega$.

\noindent
Set $\Lambda_p=\frac{1}{2}x\partial_x + \frac{1}{p-1}$.
Notice $\Lambda_p u =\left.\partial_{\omega}\right|_{\omega=1} \omega^{\frac{1}{p-1}}u(\sqrt{\omega}x)$.
Suppose $\phi=\phi[z]$ solves
\begin{align*}
-\partial_x^2\phi-|\phi|^{p-1}\phi + R = -\widetilde{\vartheta}\phi+\im \widetilde{\omega} \Lambda_p \phi + \im D_z\phi \widetilde{z},
\end{align*}
for some function $R=R[z]$ and constants $\widetilde{\vartheta}=\widetilde{\vartheta}[z]$, $\widetilde{\omega}=\widetilde{\omega}[z]$ and $\widetilde{z}=\widetilde{z}[z]$.
Then, setting
\begin{align}
\phi[\omega,z](x)&=\omega^{\frac{1}{p-1}}\phi[z](\sqrt{\omega}x),\label{def:phiomegaz}\\
R[\omega,z](x)&=\omega^{\frac{p}{p-1}}R[z](\sqrt{\omega}x),\nonumber
\end{align}
and
\begin{align}
\widetilde{\vartheta}[\omega,z]&=\omega \widetilde{\vartheta}[z],\  \widetilde{\omega}[\omega,z]=\omega \widetilde{\omega}[z],\
\widetilde{z}[\omega,z]=\omega \widetilde{z}[z],\nonumber
\end{align}
it is easy to verify
\begin{align*}
-\partial_x^2\phi[\omega,z]-|\phi[\omega,z]|^{p-1}\phi[\omega,z] + R[\omega,z] = -\widetilde{\vartheta}[\omega,z]\phi[\omega,z]+\im \widetilde{\omega}[\omega,z] \partial_\omega \phi[\omega,z] + \im D_z\phi[\omega,z] \widetilde{z}[\omega,z].
\end{align*}
Therefore, it suffices to consider  case $\omega=1$.

\noindent We also record the Taylor formula of $f$.
For $\psi>0$ and $w\in \C$, we have
\begin{align*}
f(\psi+w)=&f(\psi) + Df(\psi)w+\frac{1}{2}D^2f(\psi)(w,w) + \frac{1}{6}D^3f(\psi)(w,w,w) \\&+ \frac{1}{24}D^4 f(\psi)(w,w,w,w) + O(|w|^p).
\end{align*}
Furthermore, it is elementary to check
\begin{align}\label{eq:fdiff}
D^Nf(\psi)(w,\cdots,w)=\sum_{n=0}^N a_{N,n}|\psi|^{p-1-2n}\psi^{1-N+2n}w^{N-n}\overline{w}^n,
\end{align}
where the coefficients $a_{N,n}$ satisfies  the recurrence relation
\begin{align*}
a_{N,n}=\frac{p+1-2n}{2}a_{N-1,n-1} + \(\frac{p+3}{2}+n-N\)a_{N-1,n}.
\end{align*}
In particular, we have $a_{0,0}=1$, $a_{1,0}=\frac{p+1}{2}$ and $a_{1,1}=\frac{p-1}{2}$ and therefore
\begin{align}\label{eq:f'}
	Df(\psi)w=\frac{p+1}{2}\psi^{p-1}w+\frac{p-1}{2}\psi^{p-1}\overline{w},
\end{align}
for $\psi>0$.

\subsection{Refined profile in the 3rd   order FGR}\label{sec:refprof3rd}

This is the simplest case.
The refined profile  will be a polynomial of $z$ and $\overline{z}$
  of the form
\begin{align}&
    \phi[z]  = \phi+  \widetilde{\phi}[z] \text{ where }   \widetilde{ \phi}[z]=   \sum _{\mathbf{m}\in \mathbf{NR}_3}{z}^{\mathbf{m}} \xi _{\mathbf{m}}, \label{eq:refprof}
\end{align}
where
\begin{align*}
\mathbf{NR}_3=\{(1,0),(0,1),(2,0),(1,1),(0,2),(2,1),(1,2)\}.
\end{align*}

\begin{lemma}
There exist $\{\xi_\mathbf{m}\}_{\mathbf{m}\in \mathbf{NR}_3}\subset \sech (\kappa  x)L^2 _\rad (\R , \R)$, $G_{(3,0)}, G_{(0,3)}\in \sech (\kappa  x)L^2 _\rad (\R , \R)$ and $\lambda_{(2,1)}\in \R$ such that
\begin{align}
-\phi[z]+\im D_z\phi[z](-\im \lambda z -\im \lambda_{(2,1)} |z|^2z)=-\partial_x^2\phi[z] - f(\phi[z]) + z^3 G_{(3,0)} +\overline{z}^3 G_{(0,3)} + \widetilde{R}[z],\label{eq:rp30}
\end{align}
where $\phi[z]$ is defined by \eqref{eq:refprof} and $G_{(3,0)}$ is given by the formula \eqref{def:G30} and $G_{(0,3)}$ is given by replacing all $\xi_{ \mathbf{m}}$ in \eqref{def:G30} by $\xi_{\overline{\mathbf{m}}}$ and
\begin{align*}
\|\cosh(\kappa x) \widetilde{R}[z] \|_{L^2}\lesssim |z|^4.
\end{align*}
\end{lemma}

\begin{remark}
The left hand side of \eqref{eq:rp30} appears when one compute $e^{-\im t}\im \partial_t (e^{\im t}\phi[z])$ and assuming $\im\dot{z}= \lambda z +\lambda_{(2,1)}|z|^2z$.
\end{remark}

\begin{proof}
We substitute the ansatz \eqref{eq:refprof} into \eqref{eq:rp30} and solve the equation for the coefficients of $z^{\mathbf{m}}$ for $\mathbf{m}\in \mathbf{NR}_3$.
We first note
%
\begin{align*}
&-\phi +\sum_{\mathbf{ m} \in \mathbf{NR}} z^{\mathbf{m}} \(- 1+ (m_1-m_2)\(\lambda + \lambda_{(2,1)}|z|^2 \)\)\xi_{ \mathbf{m}}\\&
=-\phi+z(-1+\lambda)\xi_{(1,0)}+\overline{z}(1-\lambda)\xi_{(0,1)} + z^2(-1+2\lambda)\xi_{(2,0)}-|z|^2\xi_{(1,1)}+\overline{z}^2(-1-2\lambda)\xi_{(0,2)} \\&\quad+ |z|^2z \((-1 + \lambda)\xi_{(2,1)} +\lambda_{(2,1)}\xi_{(1,0)}\)
+|z|^2\overline{z}\((-1-\lambda)\xi_{(1,2)} + \lambda_{(2,1)}\xi_{(0,1)}\),
\end{align*}
where $\lambda_{(2,1)}$ to be determined.
Expanding $f(\phi[z])$  we have
\begin{align*}
f(\phi[z])&=f(\phi+\widetilde{ \phi}[z])\\&=f(\phi) + Df(\phi)\widetilde{ \phi}[z] + \frac{1}{2}D^2f(\phi)(\widetilde{ \phi}[z],\widetilde{ \phi}[z])+\frac{1}{6}D^3f(\phi)\(\widetilde{ \phi}[z],\widetilde{ \phi}[z],\widetilde{ \phi}[z]\)+O(|z|^4)\\&
=f(\phi) + \sum_{\mathbf{m}\in \mathbf{NR}_3} z^{\mathbf{m}}\phi^{p-1}\(\frac{p+1}{2} \xi_{ \mathbf{m}}+\frac{p-1}{2}{\xi}_{\overline{\mathbf{m}}}\)+\sum_{2\le |\mathbf{m}|\le 3}z^{\mathbf{m}}G_{\mathbf{ m}}+O(|z|^4) ,
\end{align*}
where
\begin{align*}
  \sum_{2\le |\mathbf{m}|\le 3}z^{\mathbf{m}}G_{\mathbf{ m}}+O(|z|^4) =  \frac{1}{2}D^2f(\phi)(\widetilde{ \phi}[z],\widetilde{ \phi}[z])+\frac{1}{6}D^3f(\phi)\(\widetilde{ \phi}[z],\widetilde{ \phi}[z],\widetilde{ \phi}[z]\)+O(|z|^4) .
\end{align*}
 The $G_{\mathbf{ m}}$ can be computed by the expansion formula \eqref{eq:fdiff}.
In particular, using
\begin{align}\label{f''}
D^2f(\phi)(w,w)=\frac{p-1}{4}\phi^{p-2}\((p+1)w^2+2(p+1) |w|^2+(p-3) \overline{w}^2\),
\end{align}
we have
\begin{align}
G_{(2,0)}&=\frac{(p-1)}{8}\phi^{p-2}\((p+1)\xi_{(1,0)}^2+2(p+1)\xi_{(1,0)}\xi_{(0,1)}+(p-3)\xi_{(0,1)}^2\)\label{def:G20}\\
G_{(1,1)}&=\frac{p-1}{4}\phi^{p-2} \((p+1)\(\xi_{(1,0)}^2+\xi_{(0,1)}^2\) +2(p-1)\xi_{(1,0)}\xi_{(0,1)}\),\label{def:G11}
\end{align}
and $G_{(0,2)}$ is given by replacing all $\xi_{\mathbf{m}}$ in $G_{(2,0)}$ by $\xi_{\overline{\mathbf{m}}}$.
This symmetry comes from the fact
\begin{align*}
\overline{D^2f(\psi)(w,w)}=D^2f(\psi)(\overline{w},\overline{w}),
\end{align*}
which follows from the explicit formula \eqref{f''}.

Using
\begin{align}\label{f'''}
D^3f(\psi)(w,w,w)=&\frac{p-1}{8} \psi^{p-3}\((p-3)(p+1)w^3+3(p-1)(p+1)w^2\overline{w}\right.\nonumber\\&\qquad\left.+3(p-3)(p+1) w \overline{w}^2+(p-5)(p-3)\overline{w}^3\),
\end{align}
as well as \eqref{f''}, we have
\begin{align}
G_{(2,1)}&=\frac{p-1}{4}\phi^{p-2}\(\xi_{(1,0)}\( (p+1)\xi_{(2,0)}+2(p+1)\xi_{(1,1)}+(p-3)\xi_{(0,2)}\)\right.\label{def:G21}\\&\qquad\qquad\left.+\xi_{(0,1)}\((p+1)\xi_{(2,0)}+2(p-1)\xi_{(1,1)}+(p+1)\xi_{(0,2)}\)\)\nonumber\\
&\quad +\frac{p-1}{16}\phi^{p-3}\( (p-1)(p+1)\xi_{(1,0)}^3+3(p-3)(p+1)\xi_{(1,0)}^2\xi_{(0,1)}\right.\nonumber\\&\qquad\qquad\left. +(3p^2-8p+13)\xi_{(1,0)}\xi_{(0,1)}^2+(p-3)(p+1)\xi_{(0,1)}^3\),\nonumber \\
G_{(3,0)}&=\frac{p-1}{4}\phi^{p-2}\((p+1)\(\xi_{(1,0)}\xi_{(2,0)}+\xi_{(1,0)}\xi_{(0,2)}+\xi_{(0,1)}\xi_{(2,0)}\)+(p-3)\xi_{(0,1)}\xi_{(0,2)}\) \label{def:G30}\\
&\quad+\frac{p-1}{48}\phi^{p-3}\((p-3)(p+1)\xi_{(1,0)}^3+3(p-1)(p+1)\xi_{(1,0)}^2\xi_{(0,1)}\right.\nonumber\\&\qquad\qquad\left.+3(p-3)(p+1)\xi_{(1,0)}\xi_{(0,1)}^2+(p-5)(p-3)\xi_{(0,1)}^3\),\nonumber
\end{align}
and $G_{(1,2)}$ and $G_{(0,3)}$ are given by replacing all $\xi_{\mathbf{m}}$ in $G_{(2,1)}$ and $G_{(3,0)}$ by $\xi_{\overline{\mathbf{m}}}$ respectively.

We want $\phi[z]$ to solve our equation \eqref{eq:rp30} up to cubic resonant terms and    a $O(z^4)$ error.  We will see that this requirement yields
  $\lambda_{(2,1)}$ and $\xi_{ \mathbf{m}}$.

\noindent The condition that the coefficients of $z$ and $\overline{z}$ are null  becomes
\begin{align}\label{def:xi10}
(-1+\lambda)\xi_{(1,0)} &=-\partial_x^2 \xi_{(1,0)} -\frac{p+1}{2}\phi^{p-1} \xi_{(1,0)} -\frac{p-1}{2}\phi^{p-1}\xi_{(0,1)}\\
(-1-\lambda)\xi_{(0,1)} &=-\partial_x^2 \xi_{(0,1)} -\frac{p-1}{2}\phi^{p-1} \xi_{(1,0)} -\frac{p+1}{2}\phi^{p-1}\xi_{(0,1)}, \label{def:xi01}
\end{align}
which can be written as
\begin{align}\label{eq:evlambda}
(H-\lambda)\begin{pmatrix}
\xi_{(1,0)}\\
\xi_{(0,1)}
\end{pmatrix}=0,
\end{align}
where
 \begin{align}\label{eq:opH}
   &   H    =  \begin{pmatrix}
-\partial_x^2+1 -\frac{p+1}{2}\phi^{p-1} & -\frac{p-1}{2}\phi^{p-1}\\
\frac{p-1}{2}\phi^{p-1} & \partial_x^2-1 +\frac{p+1}{2}\phi^{p-1}
\end{pmatrix}   =     \sigma _3 \(-\partial _x^2 +1 \) +  V   , \\&
  V :=  V (p,x)   :=    M_0 \sech ^2   \(   \frac{p-1}2 x\)  \text{ with }    M_0:= -\( \frac{p+1}{2}
    \sigma _3      +\im  \frac{p-1}{2}  \sigma _2 \)   \frac {p+1}2   .  \nonumber
\end{align}
Now, like in \cite{CM24D1} let the matrix $U$ be defined
by
\begin{align}\label{def:U}
	U = \begin{pmatrix}
		1 &
		1  \\
		\im &
		-\im      \end{pmatrix} \, , \quad
	U^{-1}= \frac 12  \begin{pmatrix}  1 &
		-\im   \\
		1 &
		\im    \end{pmatrix}   .\end{align}
Then, we have
\begin{align}\label{eq:opH1}
   & U ^{-1}\mathcal{L} U = \im  H.
\end{align}
Thus, we can choose
\begin{align*}
U\begin{pmatrix}
\xi_{(1,0)}\\
\xi_{(0,1)}
\end{pmatrix}
=\xi,
\end{align*}
where $\xi=\xi[1]=(\xi_1,\xi_2)^\intercal$ was the normalized eigenfunction of $\mathcal{L}$ associated to $\im \lambda$.
Thus, we set
\begin{align}
\xi_{(1,0)}:=\frac{1}{2}\(\xi_1 -\im \xi_2\),\label{def:xi10}\\
\xi_{(0,1)}:=\frac{1}{2}\(\xi_1+\im \xi_2 \).\label{def:xi01}
\end{align}
Recall that $\xi_1=\Re \xi_1$ and $\xi_2=\im \Im \xi_2$, so $\xi_{(1,0)}$ and $\xi_{(0,1)}$ are $\R$-valued.

\noindent We next consider the $z ^{\mathbf{m}} $  case $|\mathbf{m}|=2$.
In this case, comparing the coefficients of $z^2$ and $\overline{z}^2$,  the condition that these coefficients  are null is equivalent to
\begin{align}\label{def:xi20}
(H-2\lambda)\begin{pmatrix}
\xi_{(2,0)}\\
\xi_{(0,2)}
\end{pmatrix}
=\begin{pmatrix}
G_{(2,0)}\\
-G_{(0,2)}
\end{pmatrix}.
\end{align}
Notice  that from the expression of $G_{(2,0)}$ and $G_{(0,2)}$ given in \eqref{def:G20} and below \eqref{def:G11}, the right hand side is a known term.
Further, by Assumption \ref{ass:eigenvalues1}, $H-2\lambda$ is invertible.
Thus, we can determine $\xi_{(2,0)}$ and $\xi_{(0,2)}$  that solve the above equations, by
\begin{align*}
\begin{pmatrix}
\xi_{(2,0)}\\
\xi_{(0,2)}
\end{pmatrix}=(H-2\lambda)^{-1}\begin{pmatrix}
G_{(2,0)}\\
-G_{(0,2)}
\end{pmatrix}.
\end{align*}

\begin{remark}
Even though $G_{(2,0)}$ and $G_{(0,2)}$ are related by the symmetry of replacing $\xi_{\mathbf{m}}$ with $\xi_{\overline{\mathbf{m}}}$, we have no explicit formula to relate $\xi_{\mathbf{m}}$ and $\xi_{\overline{\mathbf{m}}}$.
This is due to the lack of explicit formula relating $\xi_1$ and $\xi_2$.
\end{remark}

The  condition that the   coefficient of $|z|^2$ is null yields the     equation
\begin{align}\label{def:xi11}
L_+\xi_{(1,1)} = G_{(1,1)},
\end{align}
where $G_{(1,1)}$ is given in \eqref{def:G11}.
Since $L_+$ is invertible (for even functions), we determine
\begin{align*}
\xi_{(1,1)}=L_+^{-1} G_{(1,1)}.
\end{align*}
Finally,   the coefficients of $|z|^2z$ and $|z|^2\overline{z}$ are null exactly if we have
\begin{align}\label{eq:3rd(2,1)}
(H-\lambda)\begin{pmatrix}
\xi_{ (2,1)}\\
\xi_{ (1,2)}
\end{pmatrix}
=\lambda_{(2,1)} \begin{pmatrix}
\xi_{(1,0)}\\
\xi_{(0,1)}
\end{pmatrix}
+
\begin{pmatrix}
G_{(2,1)}\\
-G_{(1,2)}
\end{pmatrix},
\end{align}
where $G_{(2,1)}$ and $G_{(1,2)}$ is given by \eqref{def:G21} and below \eqref{def:G30} and since we have already determined $\xi_{\mathbf{m}}$ with $|\mathbf{m}|\leq 2$, they are given terms.
Now, since $\mathrm{Ran}\(H-\lambda\)=\mathrm{ker}(H^*-\lambda)^\perp$ and $H^*=\sigma_3H\sigma_3$, we have $$\mathrm{Ran}(H-\lambda)=\left\{\begin{pmatrix}
\xi_{(1,0)}\\
-\xi_{(0,1)}
\end{pmatrix}\right\}^\perp.$$
Therefore, by $\<(\xi_{(1,0)}\ \xi_{(0,1)})^\intercal,(\xi_{(1,0)}\ -\xi_{(0,1)})^\intercal\>=1/2$, which follows from \eqref{eq:xinormliz1}, \eqref{def:xi10} and \eqref{def:xi01}, setting
\begin{align*}
\lambda_{(2,1)}=-2\<\begin{pmatrix}
\xi_{(1,0)}\\
\xi_{(0,1)}
\end{pmatrix},
\begin{pmatrix}
G_{(2,1)}\\
G_{(1,2)}
\end{pmatrix}\>
\end{align*}
we have a unique $(\xi_{ (2,1)}\ \xi_{ (1,2)})^\intercal$ which satisfies \eqref{eq:3rd(2,1)}.
Notice that since $\xi_{\mathbf{m}}$ and $G_{\mathbf{m}}$ are all $\R$-valued $\lambda_{(2,1)}$ is also $\R$-valued.

\noindent Since the only $z^3$  and $\bar{z}^3$ terms appear from the expansion of $f$, the proof is done.
\end{proof}


We set $$R[z]=z^3 G_{(3,0)} +\overline{z}^3 G_{(0,3)} + \widetilde{R}[z].$$

Adjusting the modulation parameters, we can make adjust $R[z]$ to get a similar quantity  symplectically orthogonal to the manifold of refined profiles.

\begin{proposition}[Refined Profile 3rd order FGR]\label{prop:refpropf}
There exists $\delta>0$ s.t.\ there exist $\widetilde{\omega}_{\mathcal{R}}, \widetilde{\vartheta}_{\mathcal{R}}\in C^4(D_{\C}(0,\delta),\R)$ and $\widetilde{z}_{\mathcal{R}} \in C^4(D_{\C}(0,\delta),\C)$
with  \begin{align}\label{eq:estpar}
 |\widetilde{\vartheta}_\mathcal{R}| + |\widetilde{\omega}_\mathcal{R}| +|\widetilde{z}_\mathcal{R}|  \lesssim |z|^3,
\end{align}
such that  setting $\widetilde{\vartheta}=1+\widetilde{\vartheta}_{\mathcal{R}}$, $\widetilde{\omega}=\widetilde{\omega}_{\mathcal{R}}$ and $\widetilde{z}=-\im(\lambda   +\lambda_{(2,1)} |z|^2 ) z +\widetilde{z}_{\mathcal{R}}$ and
\begin{align}\label{eq:phi_pre_gali}
R^\perp[{z}]:= \partial ^2_x\phi  [z]+ f(\phi [z]) - \widetilde{\vartheta}\phi [z]+ \im \widetilde{\omega}\partial_{\omega}\phi [z] +\im  D_{z}\phi  [z]\widetilde{z },
\end{align}
for $z_1=\Re z$ and $z_2=\Im z$, we have
\begin{align}\label{R:orth} &
\< \im {R}^\perp[z], \phi [z]\>=\< \im {R}^\perp[z],\im  \Lambda_p\phi [z]\>  =\< \im {R}^\perp[z],\im \partial_{z_{ j}}\phi[z]\> = 0,\text{   for all $j=1,2 $.}
\end{align}
Moreover, $R^\perp[z]$ can be expanded as
\begin{align*}
R^\perp[z]=z^3G_{(3,0)}^\perp + \overline{z}^3 G_{(0,3)}^\perp +\widetilde{R}^\perp[z],
\end{align*}
with $G^\perp_{(3,0)}$ and $G^\perp_{(0,3)}$ explicitly given by \eqref{G30perp} and \eqref{G03perp} below and
\begin{align}&
 \| \cosh \( \kappa  x  \)  \widetilde{R}^\perp    [z]\| _{L^2(\R ) }\lesssim |z|^4 .\label{estR}
\end{align}

 \end{proposition}

\proof
Recall \eqref{def:phiomegaz}.
For $\Theta=(\vartheta,\omega,z)$, we set $\phi[\Theta]=e^{\im \vartheta}\phi[\omega,z]$ and
\begin{align}
  \label{eq:refmanifold} \mathcal{S}_\delta:=\{\phi[\Theta]\ |\ \vartheta\in \R, |\omega-1|<\delta, \ |z|<\delta\}
\end{align}
We   claim that  $\mathcal{S}_\delta$
is a symplectic submanifold of $L^2_{\mathrm{rad}}(\R,\C)$ with the symplectic form $\<\im \cdot,\cdot\>$.
Indeed, from \eqref{def:xi10} and \eqref{def:xi01}, we have $z\xi_{(1,0)}+\bar{z}\xi_{(0,1)}=z_1\xi_1+z_2\xi_2$ ($z=z_1+\im z_2$, $z_1,z_2\in \R$) and thus
\begin{align*}
T_{e^{\im \vartheta}\phi}S_{\delta} =\mathrm{span}\{e^{\im \vartheta} \im \phi,e^{\im \vartheta} \Lambda_p\phi , e^{\im \vartheta}\xi_1, e^{\im \vartheta}\xi_2\},
\end{align*}
where $T_\psi S_{\delta}$ is the tangent space of $S_{\delta}$ at $\psi\in S_{\delta}$.
By the relation
\begin{align*}
\<\im \Lambda_p\phi,\im\phi\>=\frac{1}{2}-\frac{2}{p-1}\neq 0,\ \<\im \xi_1,\xi_2\>=\frac{1}{2},\ \<\im \xi_j,\im\phi\>=\<\im \xi_j,\Lambda_p\phi\>=0\  (j=1,2),
\end{align*}
we see that $\<\im \cdot,\cdot\>$ restricted on $T_{e^{\im \vartheta}\phi}S_{\delta}$ is nondegenerate.
Thus, $T_{e^{\im \vartheta}\phi}S_{\delta}$ is a symplectic vector space.
Further, by continuity, $\<\im \cdot,\cdot\>$ is non-degenerate restricted on  $T_{\phi[\Theta]}S_{\delta}$ for $|\omega-1|<\delta$, $|z|<\delta$ and $\delta>0$ sufficiently small.
Therefore, we conclude $S_\delta$ is a symplectic manifold.

Now, we set $e_1[\Theta]= \im \phi[\Theta]$, $f_1[\Theta]= \partial_\omega \phi[\Theta]$.
Then, we have $\<\im e_1[\Theta],f_1[\Theta]\>\neq 0$, provided $\phi[\Theta]\in S_\delta$.
We choose $e_2[\Theta]$ and $f_2[\Theta]$ from $\{e_1[\Theta],f_1\}^\perp:=\{w\in T_{\phi[\Theta]}S_{\delta}\ |\ \<\im w,e_1[\Theta]\>=\<\im w, f_1[\Theta]\>=0 \}$
so that $\<e_2[\Theta],f_2[\Theta]\>\neq 0$.
In particular, we choose $e_2$ and $f_2$ smoothly so that $e_2[\vartheta,1,0]=e^{\im \vartheta}\xi_1$ and $f_2[\vartheta, 1,0]=e^{\im \vartheta}\xi_2$.

We now set the projection $P[\Theta]$ by
\begin{align*}
P[\Theta]\psi:=\sum_{j=1}^2\frac{1}{\<\im e_j[\Theta],f_j[\Theta]\>} \(\<\im \psi,f_j[\Theta]\>e_j[\Theta]+\<\im e_j[\Theta],\psi\>f_j[\Theta]\).
\end{align*}
Then, we see that $P[\Theta]$ is a projection on to $T_{\phi[\Theta]}S_\delta$ satisfying
\begin{align}
\<\im P[\Theta]v,w\>=\<\im v,P[\Theta]w\>.\label{symporth}
\end{align}
We remark that a projection on a symplectic subspace satisfying \eqref{symporth} is unique and therefore $P[\Theta]$ actually do not depend on the choice of $\{e_j[\Theta],f_j[\Theta]\}_{j=1,2}$.

We can now set $R^\perp[z]$ by $R^\perp[z]=-\im P^\perp[0,1,z]\im R[z]$, where $P^\perp[\Theta]=1-P[\Theta]$.
Notice that from the definition of $P[\Theta]$, we have the orthogonality condition \eqref{R:orth}.
Indeed, for example the first orthogonality condition follows from
\begin{align*}
\<\im R^\perp[z],\phi[z]\>&=\<P^\perp[0,1,z]\im R[z],\phi[z]\>=\<\im P^\perp[0,1,z]\im R[z],\im \phi[z]\>=-\<R[z],P^\perp[0,1,z]\im \phi[z]\>\\&=0,
\end{align*}
where in the 3rd equality we have used \eqref{symporth} and in the last equality we have used $P[0,1,z]\im \phi[z]=\im \phi[z]$.
The other orthogonality conditions follow in a similar manner.
Now, since the range of $P[0,1,z]$ is $T_{\phi[z]}S_\delta$, we can uniquely express $-\im P[0,1,z]\im R[z]$ as
\begin{align*}
-\im P[0,1,z]\im R[z]&=-\im \(\widetilde{\vartheta}_{\mathcal{R}}[z]\im \phi[z]+\widetilde{\omega}_{\mathcal{R}}[z]\Lambda_p\phi[z]+\sum_{j=1,2}\widetilde{z}_{j,\mathcal{R}}[z]\partial_{z_j}\phi[z]\)\\&
=\widetilde{\vartheta}_{\mathcal{R}}[z] \phi[z]-\im \widetilde{\omega}_{\mathcal{R}}[z]\Lambda_p\phi[z]-\im D_z\phi[z]\widetilde{z}_{\mathcal{R}}[z],
\end{align*}
where $\widetilde{z}_{\mathcal{R}}=\widetilde{z}_{1,\mathcal{R}}+\im \widetilde{z}_{2,\mathcal{R}}$.
The estimate \eqref{eq:estpar} follows form $|R[z]|\lesssim |z|^3$.
Substituting,
\begin{align*}
R[z]&=R^\perp[z]-\im P[0,1,z]\im R[z]\\&
=R^\perp[z]+\widetilde{\vartheta}_{\mathcal{R}}[z] \phi[z]-\im \widetilde{\omega}_{\mathcal{R}}[z]\Lambda_p\phi[z]-\im D_z\phi[z]\widetilde{z}_{\mathcal{R}}[z],
\end{align*}
into \eqref{eq:rp30}, we see that $R^\perp[z]$ defined in \eqref{eq:phi_pre_gali} corresponds to $R^\perp[z]$ given above.

Finally, since $P[\Theta]$ is only $\R$-linear, we have
\begin{align*}
R^\perp[z]&=-\im P^\perp[0,1,z] \im \(z^3G_{(3,0)}+\overline{z}^3G_{(0,3)}\)+O(|z|^4)\\&
=-\im P^\perp[0,1,0]\im \(z^3G_{(3,0)}+\overline{z}^3G_{(0,3)}\)+O(|z|^4)\\&
=-\im \Re(z^3)P^\perp [0,1,0]\im(G_{(3,0)}+G_{(0,3)})+\im \Im(z^3)P^\perp[0,1,0]\(G_{(3,0)}-G_{(0,3)}\)+O(|z|^4)\\&
=z^3G_{(3,0)}^\perp + \overline{z}^3 G_{(0,3)}^\perp +O(|z|^4),
\end{align*}
where
\begin{align}
G_{(3,0)}^\perp&=-\frac{\im}{2}P^\perp [0,1,0]\im(G_{(3,0)}+G_{(0,3)})+\frac{1}{2} P^\perp[0,1,0]\(G_{(3,0)}-G_{(0,3)}\),\label{G30perp}\\
G_{(0,3)}^\perp &=-\frac{\im}{2}P^\perp [0,1,0]\im(G_{(3,0)}+G_{(0,3)})-\frac{1}{2}P^\perp[0,1,0] \(G_{(3,0)}-G_{(0,3)}\),\label{G03perp}
\end{align}
with $P^\perp[0,1,0] $ explicitly given by
\begin{align*}
P^\perp[0,1,0]\psi= \psi+\frac{1}{\<\phi,\Lambda_p\phi\>}\(\<\im \psi,\Lambda_p\phi\>\im \phi -\<\phi,\psi\>\Lambda_p\phi\) - \frac{1}{\<\im \xi_1,\xi_2\>}\(\<\im \psi,\xi_2\>\xi_1+\<\im \xi_1,\psi\>\xi_2\).
\end{align*}
\qed

\subsection{Refined profile in the 4th   order FGR}\label{sec:refprof4th}

In this case, the refined profile has the form
\begin{align}\label{eq:rf4th}
\phi[z]=\phi+\sum_{\mathbf{m}\in \mathbf{NR}_4}z^{\mathbf{m}}\xi_{ \mathbf{m}},
\end{align}
with
\begin{align*}
\mathbf{NR}_4=\mathbf{NR}_3\cup \{(3,0),(0,3),(3,1),(1,3),(2,2)\}.
\end{align*}

\begin{lemma}
There exists $\{\xi_{ \mathbf{m}}\}_{\mathbf{m}\in \mathbf{NR}_4}\subset \cosh(\kappa x)L^2_{\mathrm{rad}}(\R,\R)$ and $\lambda_{(2,1)}\in \R$ such that
\begin{align*}
-\phi[z]+\im D_z\phi[z](-\im \lambda z -\im \lambda_{(2,1)} |z|^2z)=-\partial_x^2\phi[z] - f(\phi[z]) + z^4 G_{(4,0)} +\overline{z}^4 G_{(0,4)} + \widetilde{R}[z],
\end{align*}
where $\phi[z]$ is defined by \eqref{eq:rf4th}, $G_{(4,0)}$ is given by  \eqref{def:G40} and $G_{(0,4)}$ is given by replacing all $\xi_{ \mathbf{m}}$ in \eqref{def:G40} by $\xi_{\overline{\mathbf{m}}}$ and
\begin{align*}
\|\cosh(\kappa x) \widetilde{R}[z] \|_{L^2}\lesssim |z|^p.
\end{align*}

\end{lemma}

\begin{proof}
For $\mathbf{ m}\in \mathbf{NR}_3$, $G_\mathbf{ m}$ and $\xi_{ \mathbf{m}}$ are  exactly   the same of the 3rd order case.
Further, $G_{(3,0)}$ and $G_{(0,3)}$ are given by \eqref{def:G30} and below \eqref{def:G30}.
Using
\begin{align*}
&D^4f(\psi)(w,w,w,w)=\frac{p-1}{16}\phi^{p-4}\((p-5)(p-3)(p+1)w^4+4(p-3)(p-1)(p+1)w^3\overline{w}\right.\\&\qquad\left.+6(p-3)(p-1)(p+1)w^2\overline{w}^2+4(p-5)(p-3)(p+1)w\overline{w}^3+(p-7)(p-5)(p-3)\overline{w}^4\),
\end{align*}
we have
\small
\begin{align*}
G_{(3,1)}&=\frac{p-1}{4}\phi^{p-2}\((p+1)\(\xi_{(1,0)}\(\xi_{ (2,1)}+\xi_{ (1,2)}+\xi_{(3,0)}\)+\xi_{(0,1)}\(\xi_{ (2,1)}+\xi_{(3,0)}+\xi_{(0,3)}\)\right.\right.\\&\left.\left.\qquad\qquad+\xi_{(1,1)}\(2\xi_{(2,0)}+\xi_{(0,2)}\)\)+(p-3)\(\xi_{(1,0)}\xi_{(0,3)}+\xi_{(0,1)}\xi_{ (1,2)}+\xi_{(0,2)}\xi_{(1,1)}\)\)\\&\quad+
\frac{p-1}{16}\phi^{p-3}\( (p-3)(p+1)\(\xi_{(1,0)}^2\(\xi_{(1,1)}+2\xi_{(0,2)}\)+2\xi_{(1,0)}\xi_{(0,1)}\(\xi_{(2,0)}+\xi_{(1,1)}\)+2\xi_{(0,1)}^2\xi_{(0,2)}\)\right.\\&\left.\qquad\qquad+2(p-1)(p+1)\(\xi_{(1,0)}^2\xi_{(2,0)}+\xi_{(1,0)}\xi_{(0,1)}\xi_{(1,1)}+\xi_{(0,1)}^2\xi_{(2,0)}\)\right.\\&\left.\qquad\qquad+(p-5)(p-3)\(\xi_{(0,1)}^2\xi_{(1,1)}+2\xi_{(1,0)}\xi_{(0,1)}\xi_{(0,2)}\) \)\\&\quad
+\frac{p-1}{96}\phi^{p-4}\((p-5)(p-3)(p+1)\xi_{(1,0)}^3\xi_{(0,1)} + (p-3)(p-1)(p+1)\(\xi_{(1,0)}^4+3\xi_{(1,0)}\xi_{(0,1)}^3\)\right.\\&\left.\qquad\qquad+
3(p-3)(p-1)(p+1)\(\xi_{(1,0)}^3\xi_{(0,1)}+\xi_{(1,0)}\xi_{(0,1)}^3\)\right.\\&\left.\qquad\qquad+(p-5)(p-3)(p+1)\(\xi_{(1,0)}^2\xi_{(0,1)}^2+\xi_{(0,1)}^4\)+(p-7)(p-5)(p-3)\xi_{(1,0)}\xi_{(0,1)}^3\)
\end{align*}
\normalsize
and
\small
\begin{align*}
G_{(2,2)}&=\frac{p-1}{4}\phi^{p-2}\(2(p-1)\(\xi_{(1,0)}\xi_{ (1,2)}+\xi_{(0,1)}\xi_{ (2,1)}+\xi_{(2,0)}\xi_{(0,2)}\)+2p\xi_{(1,1)}^2\right.\\&\left.\qquad\qquad+(p+1)\(2\xi_{(1,0)}\xi_{ (2,1)}+2\xi_{(0,1)}\xi_{ (1,2)}+\xi_{(2,0)}^2+\xi_{(0,2)}^2\)\)\\&
\quad +\frac{p-1}{8}\phi^{p-3}\((p-3)(p-2)\(\xi_{(1,0)}^2\xi_{(0,2)}+2\xi_{(1,0)}\xi_{(0,1)}\xi_{(1,1)}+\xi_{(0,1)}^2\xi_{(2,0)}\)\right.\\&\left.\qquad\qquad+
2(p-2)(p+1)\(\xi_{(1,0)}\xi_{(0,1)}\(\xi_{(2,0)}+\xi_{(0,2)}\)+\(\xi_{(1,0)}^2+\xi_{(0,1)}^2\)\xi_{(1,1)}\)\)\\&\quad
+\frac{p-1}{64}\phi^{p-4}\(2(p-5)(p-3)^2\xi_{(1,0)}^2\xi_{(0,1)}^2+4(p-3)^2(p+1)\(\xi_{(1,0)}^3\xi_{(0,1)}+\xi_{(1,0)}\xi_{(0,1)}^3\)\right.\\&\left.\qquad\qquad+(p-3)(p-1)(p+1)\(\xi_{(1,0)}^4+4\xi_{(1,0)}^2\xi_{(0,1)}^2+\xi_{(0,1)}^4\)\)
\end{align*}
\normalsize
and $G_{(1,3)}$ is given by replacing all $\xi_{\mathbf{m}}$ in $G_{(3,1)}$ by $\xi_{\overline{\mathbf{m}}}$.
Further,
\small
\begin{align}
G_{(4,0)}&=\frac{p-1}{4}\phi^{p-2}\((p+1)\(2\xi_{(1,0)}\(\xi_{(3,0)}+\xi_{(0,3)}\)+2\xi_{(0,1)}\xi_{(3,0)}+\xi_{(2,0)}^2+2\xi_{(2,0)}\xi_{(0,2)}\)\right.\nonumber\\&\qquad\qquad \left.+(p-3)\(2\xi_{(0,1)}\xi_{(0,3)}+\xi_{(0,2)}^2\)\)\nonumber\\&\quad+
\frac{p-1}{16}\phi^{p-3}\((p-3)(p+1)\(\xi_{(1,0)}^2\xi_{(2,0)}+2\xi_{(1,0)}\xi_{(0,1)}\xi_{(0,2)}+\xi_{(0,1)}^2\xi_{(2,0)}\)\right.\nonumber\\&\qquad\qquad \left.+(p-1)(p+1)\(2\xi_{(1,0)}\xi_{(0,1)}\xi_{(2,0)}+\xi_{(1,0)}^2\xi_{(0,2)}\)+(p-5)(p-3)\xi_{(0,1)}^2\xi_{(0,2)}\)\nonumber\\&
\quad+\frac{p-1}{384}\phi^{p-4}\((p-5)(p-3)(p+1)\xi_{(1,0)}^4+4(p-3)(p-1)(p+1)\xi_{(1,0)}^3\xi_{(0,1)}\right.\nonumber\\&\left.\qquad\qquad+6(p-3)(p-1)(p+1)\xi_{(1,0)}^2\xi_{(0,1)}^2 + 4(p-5)(p-3)(p+1)\xi_{(1,0)}\xi_{(0,1)}^3\right.\nonumber\\&\left.\qquad\qquad+(p-7)(p-5)(p-3)\xi_{(0,1)}^4\)\label{def:G40}
\end{align}
\normalsize
and $G_{(0,4)}$ is given by replacing all $\xi_{\mathbf{m}}$ in $G_{(4,0)}$ by $\xi_{\overline{\mathbf{m}}}$.
Proceeding as the previous section, we determine $\xi_{(3,0)}$ and $\xi_{(0,3)}$ by
\begin{align*}
(H-3\lambda) \begin{pmatrix}
\xi_{(3,0)}\\
\xi_{(0,3)}
\end{pmatrix}
=
\begin{pmatrix}
G_{(3,0)}\\
-G_{(0,3)}
\end{pmatrix}.
\end{align*}
By Assumption \ref{ass:eigenvalues1}, $H-3\lambda$ is invertible.
Similarly, we determine $\xi_{(3,1)}$ and $\xi_{ (1,3)}$ by
\begin{align*}
(H-2\lambda) \begin{pmatrix}
\xi_{(3,1)}\\
\xi_{(1,3)}
\end{pmatrix}
=
\begin{pmatrix}
G_{(3,1)}\\
-G_{(1,3)}
\end{pmatrix},
\end{align*}
and $\xi_{ (2,2)}$ by
\begin{align*}
L_+\xi_{ (2,2)}=G_{(2,2)},
\end{align*}
because $H-2\lambda$ and $L_+$ are invertible.

Finally, $z^4$ and $\overline{z}^4$ term appear only from the expansion of $f$, we have the conclusion.
\end{proof}

Like for the 3rd order case, by adjusting the modulation parameters, we arrive at the following proposition.


\begin{proposition}[Refined Profile 4th order FGR]\label{prop:refpropf4th}
	There exists $\delta>0$ s.t.\ there exist $\widetilde{\omega}_{\mathcal{R}}, \widetilde{\vartheta}_{\mathcal{R}}\in C^4(D_{\C}(0,\delta),\R)$ and $\widetilde{z}_{\mathcal{R}} \in C^4(D_{\C}(0,\delta),\C)$
	with  \begin{align}\label{eq:estpar4}
		|\widetilde{\vartheta}_\mathcal{R}| + |\widetilde{\omega}_\mathcal{R}| +|\widetilde{z}_\mathcal{R}|  \lesssim |z|^4,
	\end{align}
	such that setting $\widetilde{\vartheta}=1+\widetilde{\vartheta}_{\mathcal{R}}$, $\widetilde{\omega}=\widetilde{\omega}_{\mathcal{R}}$ and $\widetilde{z}= -\im(\lambda   +\lambda_{(2,1)} |z|^2 ) z +\widetilde{z}_{\mathcal{R}}$ and
	\begin{align}\label{eq:phi_pre_gali4}
		R^\perp[{z}]:= \partial ^2_x\phi  [z]+ f(\phi [z]) - \widetilde{\vartheta}\phi [z]+ \im \widetilde{\omega}\partial_{\omega}\phi [z] +\im  D_{z}\phi  [z]\widetilde{z },
	\end{align}
	for $z_1=\Re z$ and $z_2=\Im z$, we have
	\begin{align*} &
		\< \im {R}^\perp[z], \phi [z]\>=\< \im {R}^\perp[z],\im  \Lambda_p\phi [z]\>  =\< \im {R}^\perp[z],\im \partial_{z_{ j}}\phi[z]\> = 0,\text{   for all $j=1,2 $,}
	\end{align*}
	and
	\begin{align*}
		R^\perp[z]=z^4G_{(4,0)}^\perp + \overline{z}^4 G_{(0,4)}^\perp +\widetilde{R}^\perp[z],
	\end{align*}
	with
\begin{align}&
 \|\cosh(\kappa x)\widetilde{R}^\perp[z]\|_{L^2}\lesssim |z|^p .\label{estR4th}
\end{align}

	

 \end{proposition}
\proof
The proof is similar to that  of Proposition \ref{prop:refpropf}.
We note that the expression of $G_{(4,0)}^\perp$ and $G_{(0,4)}^\perp$ is given by \eqref{G30perp} and \eqref{G03perp} with replacing $(3,0)$ and $(0,3)$ with $(4,0)$ and $(0,4)$ respectively.

\qed

\section{Modulation, continuation argument and proof of Theorem \ref{thm:asstab}}\label{sec:mod}

Recall that $\phi[\omega,z]$ is given in \eqref{def:phiomegaz} from $\phi[z]$, which we have constructed for $n=3,4$.
 By standard arguments, which we skip, we have the following \textit{modulation}.

\begin{lemma}[Modulation]\label{lem:mod1}
  There exist an $\delta _0 >0$ and functions
$\omega   \in C^1(   \mathcal{U} (1,\delta _0   ),  (0,\infty)) $ and
$\vartheta \in C^1( \mathcal{U} (1,\delta _0   ), \R /2\pi\Z ) $  and $z \in C^1( \mathcal{U} (1,\delta _0   ), \C ) $
 such that for  any $u \in \mathcal{U} (1,\delta _0   )$
\small \begin{align}\label{61}
  & \eta (u):=e^{-\im \vartheta( u) } u -  \phi [\omega( u) , z(u)]    \text{ satisfies }
    \\& \nonumber
\< \eta (u),\im\phi [\omega ( u), {z}( u)]\>=\< \eta (u), \partial_{\omega}\phi [\omega( u), {z}( u)]\>  =\< \eta (u),\partial_{z_{ j}}\phi[\omega ( u), {z}( u)]\> = 0,\text{   for all $j=1,2 $.}
\end{align}
 \normalsize
Furthermore we have the identities  $\omega  ({\phi}_{\omega })=\omega $,  $\vartheta  (  e^{\im \vartheta _0}  u)    = \vartheta  (    u)   + \vartheta _0 $   and $\omega  (  e^{\im \vartheta _0}  u)    = \omega  (    u)    $ and  $z  (  e^{\im \vartheta _0}  u)    = z  (    u)    $.
\end{lemma}
\qed

We    now write the ansatz
\begin{align}\label{eq:ansatz}
  u = e^{\im \vartheta} \(  \phi [\omega  , z ]+ \eta \) .
\end{align}
Exactly like in \cite{CM24D1}, we have  (recall that here $\omega _0=1$) \begin{align}
  \label{eq:oertstab1} |\omega -1|+ |z|+ \| \eta \| _{H^1}\lesssim \sqrt{\delta } \text{   for all values of time}.
\end{align}
 Following verbatim  \cite{CM24D1}, the proof of Theorem \ref{thm:asstab} is obtained   from  the following continuation argument.

\begin{proposition}\label{prop:continuation}
There exists $C>1$ such that for any $A\geq C$ and $\epsilon>0$ satisfying $\log \epsilon^{-1}>CA$, there exists $\delta>0$ such that if $u_0\in D_{H^1_{\mathrm{rad}}}(\phi,\delta)$ and
\begin{align}
  \label{eq:main2} \| \eta \| _{L^2(I, \Sigma _A )} +  \| \eta \| _{L^2(I, \widetilde{\Sigma}  )} + \| \dot \Theta - \widetilde{\Theta} \| _{L^2(I  )} + \| z ^{n } \| _{L^2(I  )}\le \epsilon
\end{align}
holds  for $I=[0,T]$ for some $T>0$
then in fact for $I=[0,T]$    inequality   \eqref{eq:main2} holds   for   $\epsilon$ replaced by $  \epsilon/2 $.
\end{proposition}
Notice that this implies that in fact the result is true for $I=\R _+$.   We will split the proof of Proposition \ref{prop:continuation} in a number of partial results obtained assuming the hypotheses of Proposition  \ref{prop:continuation}.
\begin{proposition}\label{prop:modpar} We have
  \begin{align}&
  \label{eq:modpar1}  \|\dot \vartheta -\widetilde{\vartheta} \|  _{L^1(I  )} + \|\dot \omega -\widetilde{\omega} \|  _{L^1(I  )}\lesssim  \epsilon ^2 ,\\&  \label{eq:modpar2}   \|\dot z -\widetilde{z} \|  _{L^2(I  )} \lesssim \sqrt{ \delta}\epsilon  ,\\&  \label{eq:modpar3}   \|\dot z \|  _{L^\infty(I  )} \lesssim  \sqrt{\delta} .
\end{align}

\end{proposition}
\proof The proof is verbatim the same as in \cite{CM24D1} and is a consequence of Lemma \ref{lem:lemdscrt}. \qed

The following lemma can be proved like Lemma 4.1 in \cite{CM24D1}.
\begin{lemma}\label{lem:lemdscrt} We have the estimates
 \begin{align} \label{eq:discrest1}
   &|\dot \vartheta -\widetilde{\vartheta} | + |\dot \omega -\widetilde{\omega} |  \lesssim  \(|z|^n+    \|   \eta \| _{\widetilde{\Sigma}}     \)   \|   \eta \| _{\widetilde{\Sigma}},    \\&  |\dot z -\widetilde{z} | \lesssim   \(|z|^{n-1} +    \|   \eta \| _{\widetilde{\Sigma}}     \)  \|   \eta \| _{\widetilde{\Sigma}}   .   \label{eq:discrest2}
\end{align}
\end{lemma}
Notice that the error  $|z|^2\|\eta\|_{\widetilde{\Sigma}}$ in (4.4) of \cite{CM24D1} is improved to $|z|^n\|\eta\|_{\widetilde{\Sigma}}$ in \eqref{eq:discrest1} due to the error estimates given in Propositions \ref{prop:refpropf} and \ref{prop:refpropf4th} instead of Proposition 3.1 of \cite{CM24D1}.
For \eqref{eq:discrest2}, by the same reason we have $|z|^{n-1}$ instead of $|z|$.

We will prove the following, whose proof is similar to the analogous one in  \cite{CM24D1}.
\begin{proposition}[Fermi Golden Rule (FGR) estimate]\label{prop:FGR}
 We have
 \begin{align}\label{eq:FGRint}
  \| z^{n }\|_{L^2(I)}\lesssim  A^{-1/2} \epsilon .
 \end{align}
 \end{proposition}

\begin{proposition}[Virial Inequality]\label{prop:1virial}
 We have
 \begin{align}\label{eq:sec:1virial1}
   \| \eta \| _{L^2(I, \Sigma _A )} \lesssim  A \delta + \| z^{n }\|_{L^2(I)} + \| \eta \| _{L^2(I, \widetilde{\Sigma}   )} + \epsilon ^2  .
 \end{align}
 \end{proposition}
\proof The proof which is verbatim the same of the analogous proof in  \cite{CM24D1}, with the only difference that
in \cite{CM24D1}, for the estimate $\|\cosh(\kappa x)R^\perp[z]\|_{L^2}\lesssim |z|^n$ there is a $|z|^2$ instead of the $|z|^n$ for $n=3,4$ we have here. This accounts for the $ z   ^n$ in \eqref{eq:sec:1virial1} instead of the $ z  ^2$ in the analogous inequality in  \cite{CM24D1}.

\qed

\begin{proposition}[Smoothing Inequality]\label{prop:smooth11}
 We have
 \begin{align}\label{eq:sec:smooth11}
   \| \eta \| _{L^2(I, \widetilde{\Sigma}   )} \lesssim  o _{A^{-1}} (1)  \epsilon    .
 \end{align}
 \end{proposition}
\proof  Using Propositions \ref{prop:KrSch} and \ref{lem:smooth111}, the proof is verbatim that in \cite{CM24D1} by setting $B=A^{1/3}$.
\qed

\section{The Fermi Golden Rule: proof of Proposition \ref{prop:FGR}}\label{sec:fgr}
Substituting the ansatz \eqref{eq:ansatz} in \eqref{eq:nls1}, we have
\small
\begin{align}\label{eq:nls4}
	\im \dot \eta   &=   \(-\partial_x^2+1 -Df(\phi)\)\eta  +  ( \dot{\vartheta}-\widetilde{\vartheta} +\widetilde{\vartheta_{\mathcal{R}}} +\omega-1 ) \eta  -   \im  D_\Theta \phi [\Theta]  (\dot \Theta  -\widetilde{\Theta})   -  \( D f( \phi  [\omega  , z ]  ) - D f( \phi     )\) \eta
	\\&- \( f( \phi [\omega  , z ] + \eta  )  -  f( \phi  [\omega  , z ]  ) -D f( \phi  [\omega  , z ]  ) \eta\) - R^\perp[\omega  , z ],  \nonumber
\end{align}\normalsize
where $\Theta=(\vartheta,\omega,z)$, $\widetilde{\Theta}=(\widetilde{\vartheta},\widetilde{\omega},\widetilde{z})$ and $\phi[\Theta]=e^{\im \vartheta} \phi[\omega,z]$.
In analogy to  \cite[Sect. 5]{CM24D1} and the literature, we set
 $$g_n=(g_{(n,0)},\ g_{(0,n)})^\intercal \in L^\infty(\R,\R^2)\setminus\{0\},$$ where, for $H$   given by \eqref{eq:opH}, we have
\begin{align}
	\label{eq:eqsatg2} H g_n= n \lambda (p) g_n.
\end{align}
%
Existence of  $g_n $ follows from   Krieger and Schlag \cite[Lemma 6.3]{KrSch}, also Buslaev and Perelman \cite{BP1}.

\noindent We define the FGR constant (which depends on the power $p$ in our equation)
\begin{equation}\label{eq:fgrgamma}
	\gamma_n (p)  :=     \<    G_n , g_n  \>  ,
\end{equation}
where $G_n=(G_{(n,0)} , \ G_{(0,n)})^\intercal $.
We note that we have $\<G_n^\perp,g_n\>=\<G_n,g_n\>$ because $G_n^\perp$ is a projection of $G_n$ to the symplectic orthogonal complements of the generalized kernel of $H$.

\noindent For $p\in F_n$  we can normalize $g_n$ to have  $\gamma_n(p)=1$.

%

  For $\chi\in C_0^\infty(\R,\R)$, $1_{[-1,1]}\leq \chi\leq 1_{[-2,2]}$ and $\chi_A=\chi(\cdot/A)$,  we set
\begin{align}\label{eq:FGRfunctional}
\mathcal{J}_{\mathrm{FGR}}:=-\< \im    {\eta},\chi_A \(  {z}^{n } g_{(n,0)}+ \overline{{z}}^{n } g_{(0,n)} \)      \> .
\end{align}

Then we have the following.

\begin{lemma}\label{lem:FGR1}
We have
\begin{align}
\left|\dot{\mathcal{J}}_{\mathrm{FGR}}  - \dot{\mathcal{I}}_{\mathrm{FGR}}  -|z| ^{2n}     \right |    \lesssim  A ^{-1/2}   \(  |z|^{2n}   + \| \eta \| _{\Sigma _A}^{2}+ \| \eta \| ^{2}_{\widetilde{\Sigma}}  \)    , \label{eq:lem:FGR11}
\end{align}
where
\begin{align*}
	\mathcal{I}_{FGR}=\frac{1}{2\lambda}\Im\(\<G_{(n,0)}^\perp,g_{(0,n)}\> z^{2n} -\<G_{(0,n)}^\perp,g_{(n,0)}\>\overline{z^{2n}}\).
\end{align*}
\end{lemma}
\proof
Differentiating $\mathcal{J}_{\mathrm{FGR}}$, we have \small
\begin{align*}
\dot{\mathcal{J}}_{\mathrm{FGR}}=&
-\< \im  \dot{  \eta },\chi_A \(  {z}^{n } g_{(n,0)}+ \overline{{z}}^{n } g_{(0,n)} \)      \>
-\< \im    \eta , \chi_A \( n z  ^{n-1}  \widetilde{z}  g_{(n,0)} +  n\overline{z ^{n-1}  \widetilde{z}}  g_{(0,n)}\)            \>   \\&- \<  \im     \eta , \chi_A \( nz ^{n-1}(\dot{z}  - \widetilde{z}) g_{(n,0)}+  n\overline{z ^{n-1} (\dot{z}  - \widetilde{z})} g_{(0,n)} \)            \>
 \\  =:&A_1+A_2+A_3   .\nonumber
\end{align*}
\normalsize
We start from bounding error terms.
We will frequently be using the following estimate,
\begin{align*}
	\|\chi_A\eta\|_{L^1}\lesssim A^{3/2}\|\eta\|_{\Sigma_A},
\end{align*}
which can be easily derived from H\"older's inequality.
For $A_3$, we have
\begin{align*}
	|A_3|\lesssim \|\eta\chi_A\|_{L^1} |z|^{n-1} |\dot{z}-\widetilde{z}|\lesssim A^{3/2}\delta^{n/2}\|\eta\|_{\widetilde{\Sigma}} \|\eta \|_{\Sigma_A}\lesssim A^{-1}\(\| \eta \| _{\Sigma _A}^{2}+ \| \eta \| ^{2}_{\widetilde{\Sigma}} \).
\end{align*}

Next, for $A_2$, by the definition of $\widetilde{z}$, given in Propositions \ref{prop:refpropf} and \ref{prop:refpropf4th},
\begin{align*}
	A_2=&\<\eta,\chi_A \(z^n n\lambda g_{(n,0)} -\overline{z}^n n\lambda g_{(0,n)}\) \>\\&
	+\<\eta,\chi_A\( z^{n-1} \(n\lambda_{(2,1)}|z|^2-\im\widetilde{z}_{\mathcal{R}}\)g_{(n,0)} -\overline{z^{n-1}\(n\lambda_{(2,1)}|z|^2-\im\widetilde{z}_{\mathcal{R}}\)}\)\>\\=&A_{21}+A_{22}.
\end{align*}
There will be a cancellation between $A_{21}$ and a term coming from the expansion of $A_1$.
For $A_{22}$,
\begin{align*}
	|A_{22}|\lesssim \|\chi_A\eta\|_{L^1}  |z|^n \delta^{1/2}\lesssim A^{-1}\(|z|^n + \|\eta\|_{\Sigma_A}\).
\end{align*}
We now expand $A_1$ using \eqref{eq:nls4},
\small
\begin{align*}
	A_1=&-\<(-\partial_x^2+1-Df(\phi))\eta,\chi_A \(  {z}^{n } g_{(n,0)}+ \overline{{z}}^{n } g_{(0,n)} \)      \>+\<\im D_{\Theta}\phi[\Theta](\dot{\Theta}-\widetilde{\Theta}), \chi_A \(  {z}^{n } g_{(n,0)}+ \overline{{z}}^{n } g_{(0,n)} \)      \>\\&
	-(\dot{\vartheta}-\widetilde{\vartheta}+\widetilde{\vartheta}_{\mathcal{R}}+\omega-1)\<\eta,\chi_A \(  {z}^{n } g_{(n,0)}+ \overline{{z}}^{n } g_{(0,n)} \)      \>\\&+\<\(Df(\phi[\omega,z])-Df(\phi)\)\eta,\chi_A \(  {z}^{n } g_{(n,0)}+ \overline{{z}}^{n } g_{(0,n)} \)      \>\\&+\<f(\phi[\omega,z]+\eta)-f(\phi[\omega,z])-Df(\phi[\omega,z]),\chi_A \(  {z}^{n } g_{(n,0)}+ \overline{{z}}^{n } g_{(0,n)} \)      \>\\&+\<R^\perp[\omega,z],\chi_A \(  {z}^{n } g_{(n,0)}+ \overline{{z}}^{n } g_{(0,n)} \)      \>\\
	=:&A_{11}+A_{12}+A_{13}+A_{14}+A_{15}+A_{16}.
\end{align*}\normalsize
For $j=2,3,4,5$, $A_{1j}$ are error terms.
Indeed, by \eqref{eq:discrest1} and \eqref{eq:discrest2},
\begin{align*}
	|A_{12}|\lesssim |\dot{\Theta}-\widetilde{\Theta}| |z|^n\lesssim \delta^{1/2} \|\eta\|_{\widetilde{\Sigma}}|z|^n\lesssim A^{-1}\(\|\eta\|_{\widetilde{\Sigma}}+|z|^{2n}\).
\end{align*}
By \eqref{eq:estpar}, \eqref{eq:estpar4}, \eqref{eq:oertstab1} and \eqref{eq:discrest1},
\begin{align*}
	|A_{13}|\lesssim \delta^{1/2} \|\chi_A\eta\|_{L^1}|z|^n\lesssim A^{-1} \(\|\eta\|_{\Sigma_A}^2+|z|^{2n}\).
\end{align*}
By Taylor expansion,
\begin{align*}
	|A_{14}|\lesssim \delta^{1/2} \|\eta\|_{\widetilde{\Sigma}} |z|^n\lesssim A^{-1}\(\|\eta\|_{\widetilde{\Sigma}}^2+|z|^{2n}\)
\end{align*}
Similarly, by Taylor expansion,
\begin{align*}
	|A_{15}|\lesssim \|\eta\|_{L^\infty} \|\chi_A\eta\|_{L^1} |z|^n\lesssim A^{-1} \(\|\eta\|_{\Sigma_A}^2+|z|^{2n}\).
\end{align*}
For $A_{11}$, by \eqref{eq:f'},\small
\begin{align*}
A_{11}=&-\<-\eta''+\eta,\chi_A \(  {z}^{n } g_{(n,0)}+ \overline{{z}}^{n } g_{(0,n)} \)      \> +\<\frac{p+1}{2}\phi^{p-1}\eta +\frac{p-1}{2}\phi^{p-1}\overline{\eta},\chi_A \(  {z}^{n } g_{(n,0)}+ \overline{{z}}^{n } g_{(0,n)} \)      \>\\&
=-\<\eta,\chi_A\(z^n (-g_{(n,0)}''+g_{(n,0)})+\overline{z}^n(-g_{(0,n)}''+g_{(0,n)})\)\>-\<\eta\chi_A''+2\eta'\chi_A',  {z}^{n } g_{(n,0)}+ \overline{{z}}^{n } g_{(0,n)}       \>\\&
\quad+\<\eta, \chi_A\(z^n\(\frac{p+1}{2}\phi^{p-1}g_{(n,0)}+\frac{p-1}{2}\phi^{p-1}g_{(0,n)}\) +\overline{z}^n\(\frac{p+1}{2}\phi^{p-1}g_{(0,n)}+\frac{p-1}{2}\phi^{p-1}g_{(n,0)}\)\)\>\\&
	=-\<\eta,\chi_A\(z^n (Hg_n)_1-\overline{z^n} (Hg_n)_2\)\>-\<\eta\chi_A''+2\eta'\chi_A', {z}^{n } g_{(n,0)}+ \overline{{z}}^{n } g_{(0,n)}\>\\
=:&A_{111}+A_{112},
\end{align*}\normalsize
where $(Hg_n)_j$ is the $j$th component of $Hg_n$.
By the definition of $g_n$, we have
\begin{align*}
	A_{111}+A_{21}=0.
\end{align*}
For $A_{112}$, we have
\begin{align*}
	|A_{112}|&\lesssim A^{-2} \|1_{[-2A,2A]}\eta\|_{L^1}+A^{-1}\|1_{[-2A,2A]}\eta'\|_{L^1} |z|^n\\&\lesssim A^{-1/2} \(\|\eta\|_{\Sigma_A}^2+|z|^{2n}\).
\end{align*}
Finally, we consider $A_{16}$, which produces the main term.
\begin{align*}
	A_{16}=&\<z^n G_{(n,0)}^\perp +\overline{z^n}G_{(0,n)}^\perp,  {z}^{n } g_{(n,0)}+ \overline{{z}^{n }} g_{(0,n)}\>\\&
	+\<(\chi_A-1)\(z^n G_{(n,0)}^\perp +\overline{z^n}G_{(0,n)}^\perp\),  {z}^{n } g_{(n,0)}+ \overline{{z}^{n }} g_{(0,n)}\>\\&
	+\<z^n \(G_{(n,0)}[\omega]-G_{(n,0)}\)+\overline{z^n} \(G_{(0,n)}[\omega]-G_{(0,n)}\),\chi_A\({z}^{n } g_{(n,0)}+ \overline{{z}^{n }} g_{(0,n)}\)\>\\&
	+\<\widetilde{R}^\perp[\omega,z],\chi_A\({z}^{n } g_{(n,0)}+ \overline{{z}^{n }} g_{(0,n)}\)\>\\
	=&A_{161}+A_{162}+A_{163}+A_{164}.
\end{align*}
Now, since $G_{(n,0)}^\perp$ and $G_{(0,n)}^\perp$ are exponentially decaying functions, we have
\begin{align*}
	|A_{162}|\lesssim e^{-\kappa A}|z|^{2n}\lesssim A^{-1} |z|^{2n}.
\end{align*}
Next, by $G_{\mathbf{m}}[\omega](x)=\omega^{\frac{p}{p-1}}G_{\mathbf{m}}(\omega^{1/2}x)$ and \eqref{eq:oertstab1},
\begin{align*}
	|A_{163}|\lesssim |\omega-1| |z|^{2n}\lesssim A^{-1}|z|^{2n}.
\end{align*}
For $A_{164}$, by Propositions \ref{prop:refpropf} and \ref{prop:refpropf4th},
\begin{align*}
	|A_{164}|\lesssim |z|^{2n+(p-4)}\lesssim A^{-1}|z|^{2n}.
\end{align*}
For $A_{16}$, because we have normalized $\gamma(p)=1$, we have
\begin{align*}
	A_{16}&=|z|^4+\Re\(z^{2n}\<G_{(n,0)}^\perp,g_{(0,n)}\>+\overline{z}^{2n}\<G_{(0,n)}^\perp,g_{(n,0)}\>\)\\&=|z|^4 +A_{161}.
\end{align*}
Now, by computing $\frac{d}{dt}z^{2n}$, we have
\begin{align*}
	z^{2n}=\frac{\im}{2n\lambda} \frac{d}{dt}(z^{2n}) -\frac{\im z^{2n-1}}{\lambda}\(\dot{z}-\widetilde{z} -\im \lambda_{(2,1)} |z|^2z +\widetilde{z}_{\mathcal{R}}\).
\end{align*}
Thus,
\small
\begin{align*}
	A_{161}=&\frac{d}{dt}\(\frac{1}{2\lambda}\Im\(\<G_{(n,0)}^\perp,g_{(0,n)}\> z^{2n} -\<G_{(0,n)}^\perp,g_{(n,0)}\>\overline{z^{2n}}\)\)\\&+\Re\(\(-\frac{\im z^{2n-1}}{\lambda}\(\dot{z}-\widetilde{z} -\im \lambda_{(2,1)} |z|^2z +\widetilde{z}_{\mathcal{R}}\)\)\<G_{(n,0)}^\perp,g_{(0,n)}\>\)\\&+\Re\(\overline{\(-\frac{\im z^{2n-1}}{\lambda}\(\dot{z}-\widetilde{z} -\im \lambda_{(2,1)} |z|^2z +\widetilde{z}_{\mathcal{R}}\)\)}\<G_{(0,n)}^\perp,g_{(n,0)}\>\)\\
	=:&\frac{d}{dt}\mathcal{I}_{FGR} +A_{1611}+A_{1612}.
\end{align*}
\normalsize
By \eqref{eq:estpar}, \eqref{eq:estpar4} and \eqref{eq:discrest2}, we have
\begin{align*}
	|A_{1611}|+|A_{1612}|\lesssim \delta^{1/2} |z|^n
\(|z|^n+\|\eta\|_{\widetilde{\Sigma}}\)\lesssim A^{-1}\(|z|^{2n} +\|\eta\|_{\widetilde{\Sigma}}^2\).
\end{align*}
Therefore, we have the conclusion.

\qed

\noindent
\textit{Proof of Proposition \ref{prop:FGR}.
}
Integrating \eqref{eq:lem:FGR11}  we obtain $\| z^n \| _{L^2(I)}   ^2 \lesssim \sqrt{\delta} + A ^{-1}\epsilon ^2$ yielding \eqref{eq:FGRint}.

\qed

\section{Proof of Proposition \ref{lem:smoothest1}} \label{sec:smooth}

It is enough to prove Proposition \ref{lem:smoothest1} for the case $\omega=1$.
By \eqref{eq:opH},
it is enough to prove  the following, with $P_c$ the projection on the continuous spectrum component for $H$, the operator in \eqref{eq:opH}.
\begin{proposition} \label{lem:smoothest} For  $S>3/2$ and $\tau >1/2$ there exists a constant $C(S,\tau )$ such that we have
 \begin{align}&   \label{eq:smoothest1}   \left \|   \int   _{0} ^{t   }e^{-\im (t-t') H  }P_c g(t') dt' \right \| _{L^2( \R ,L^{2,-S}(\R ))  } \le C(S,\tau ) \|  g \| _{L^2( \R , L^{2,\tau}(\R ) ) }.
\end{align}
\end{proposition}
\proof By the argument in \cite{CM24D1}, Proposition \ref{lem:smoothest} is a consequence of Lemma  \ref{lem:LAP} written below.
\qed

The following  simplifies the proof of the analogous result in \cite{CM24D1}, which in  \cite{CM24D1} was valid for $p$ close to 3.

\begin{lemma} \label{lem:LAP} For  $S>3/2$ and $\tau >1/2$  we have
 \begin{align}&   \label{eq:LAP1}   \sup _{E \in   \R  } \|   R ^{\pm }_{H }(E ) P_c \| _{L^{2,\tau}(\R ) \to L^{2,-S}(\R )} <\infty,
\end{align}
where $ R ^{+ }_{H }(E )$ resp. $ R ^{- }_{H }(E )$ are extensions from above resp. below of the resolvent $ R  _{H }(E )$ on the real axis, see below.

\end{lemma}
\proof Let us consider the  scalar  Schr\"odinger operator
\begin{align} \label{eq:hp}
  h _p= -\partial _x ^2   +  \sech ^2 \(  \frac{p-1}{2} x \) .
\end{align}
   In \cite{CM2109.08108} it is noticed that
\begin{align*}&    \sup _{E \in   \R  } \|   R ^{\pm }_{h _p}(E )  \| _{L^{2,\tau}(\R ) \to L^{2,-S}(\R )} <\infty  .
\end{align*}
This is equivalent to
\begin{align}&   \label{eq:LAPh}   \sup _{E \in   \R  } \|   R ^{\pm }_{ \sigma _3 (h_p+1)  }(E )  \| _{L^{2,\tau}(\R ) \to L^{2,-S}(\R )} <\infty  .
\end{align}
We can write, see \eqref{eq:opH},
\begin{align*}
   H= \sigma _3 (h _p +1)+ \mathbf{v}    \text{ with }   \mathbf{v}= \( M_0-\sigma _3 \)   \sech ^2 \(  \frac{p-1}{2} x \) .
\end{align*}
We can factor the vector potential
 \begin{align}\label{eq:factor1}
     \mathbf{v} =B^*A \text{  with }   B^* =  \langle x
\rangle ^{  S}   {V}   \text{  and }A=    \langle x
\rangle ^{- S}    .
 \end{align}
Let us take $a>0$ such that       $0< \lambda (p) < a <1$. Then,  for $\Im z >0$ and  for $ Q_0 (z) =A R_{\sigma _3 (h _p+1)    } (z ) B^*$, we have
 \begin{align}\label{eq:factor3}
  A R_H(z)= \( 1+  Q_0 (z) \) ^{-1}A R_{\sigma _3 (h _p +1)    } (z )
 \end{align}
 The function $ Q_0 (z)$ extends as  an element   $ Q_0^+\in C^0(\overline{\C} _+ \backslash \sigma _p (H), \mathcal{L}  (L^2 ))$ and furthermore for some constant $C_1$
\begin{align*}
  \sup _{z\in \R \backslash (-a,a)}\|   \( 1+  Q_0 ^+  (z) \) ^{-1} \| _{L^2(\R , \C^2 ) \to L^2(\R , \C^2 )}<C _1<\infty
\end{align*}
because   singularities would correspond  to eigenvalues of $H$  and  resonances of $H$ at $\pm 1$  which do not exist in $\R \backslash (-a,a)$.  In fact singularities would correspond  by Fredholm alternative
to having $\ker \( 1+  Q_0 ^+  (z) \) \neq 0$. For $z\neq \pm 1$ by exactly the same theory of the 3 dimensional case  discussed in \cite{CPV}  we would get $\ker (H-z)\neq 0$, which  we know does not exist for $|z|\ge a$.  Case $z=  1$   or $-1$ is discussed in \S 9 \cite{CM24D1}  where it is shown that this would make 1 or $-1$ (and so both, by the symmetries of $H$) into a resonance, which we are excluding.

\noindent Then, for some constant $C_2 $ for $E \in \R \backslash (-a,a)$   we have the following,  yielding \eqref{eq:LAP1} for $E\in \R \backslash (-a,a)$, \small
\begin{align*} &
   \|   R ^{+ }_{H  }(E )   \| _{L^{2,\tau}(\R ) \to L^{2,-S}(\R )}  = \|  A R ^{+ }_{ H }(E )  \< \cdot  \> ^{-\tau} \| _{L^2(\R , \C^2 ) \to L^2(\R , \C^2 )} \\& \le  \sup _{z\in \R \backslash (-a,a)}\|   \( 1+  Q_0^+ (z) \) ^{-1} \| _{L^2(\R , \C^2 ) \to L^2(\R , \C^2 )}  \sup _{z\in \R \backslash (-a,a)} \|   R ^{+ }_{ \sigma _3 (h_p +1)  }(z )   \| _{L^{2,\tau}(\R ) \to L^{2,-S}(\R )} \le C_2<+\infty .
\end{align*}\normalsize
  For the trivial proof in the case $E\in   (-a,a)$  see \cite{CM24D1}. This gets \eqref{eq:LAP1} for the $+$ sign.
\qed

\section{Proof of Proposition \ref{prop:nonres}}\label{sec:prop:nonres}

The sentence in line \eqref{eq:nores111} is equivalent to the following one,
\begin{equation}\label{eq:nores111H}
  \text{ there is   no nonzero bounded  solution of $H u=    u$. }
\end{equation}
We will set
\begin{align}\label{eq:versors}
  e_1:= \begin{pmatrix}
 1   \\ 0
\end{pmatrix} \text{   and } e_2:= \begin{pmatrix}
 0   \\ 1
\end{pmatrix} .
\end{align}
Given two (column) functions $f,g:\R \to \C ^2$,  using the row column product, we consider the Wronskian
\begin{align*}
  W[f ,g](x) := f'(x) ^ \intercal g(x)- f (x) ^ \intercal g '(x).
\end{align*}
Before discussing about the Jost functions of $H$,  \cite{BP1,KrSch}, we state the following elementary fact.

\begin{lemma}\label{lem:harm} Let $k= \alpha+\im \beta$. Then we have
   \begin{align}\label{eq:lem:ham}
      \beta  -\Re    \sqrt{2+ \alpha^2-\beta^2+2\im \alpha\beta  } \le 0 \text{ for }     -\sqrt{2}  \le   \beta   \le 1  \text
      { and all } \alpha\in \R .
   \end{align}
   \end{lemma}
\proof  Since the function in the left hand side in  \eqref{eq:lem:ham}  is harmonic and is bounded in the strip with limit $-\infty$ at infinity, and since it is obviously negative for $\beta =-\sqrt{2}$, it is enough to prove the above bound for   $\beta =1. $
So we are reduced to  proving
\begin{align*}
 F(\alpha ):= 1  -\Re    \sqrt{1+ \alpha^2 +2\im \alpha   } \le 0 \text{ for }      \alpha \ge 0 .
\end{align*}
We have $F(0)=0$     and
\begin{align*}   F' (\alpha )  &= -\frac{d}{d\alpha} \( \sqrt{(1+\alpha ^2)^2+4\alpha ^2} \cos \(  \arctan \(  \frac{2\alpha }{1+\alpha ^2} \)  \)  \)   \\& = - \frac{  \sin  \(  \arctan \(  \frac{2\alpha }{1+\alpha ^2} \)  \) }{ \sqrt{(1+\alpha ^2)^2+4\alpha  ^2 }}  \( \alpha ^4+ 4 \alpha ^2+1\)<0 \text{ for } \alpha>0.
\end{align*}
Hence  $ F  (\alpha )\le 0$  for all $\alpha \in \R$ and Lemma  \ref{lem:harm}  is proved. \qed

It is well known that $H  $ has Jost functions, discussed in \cite{BP1,KrSch}, which we review here. Notice that we fix a $p_1 >1$   and in the remainder of this section we focus only on   $p\in (p_1, +\infty )$.
\begin{lemma}
  \label{lem:f3}   For any $p>p_1$, for $0< \gamma < p-1$ and for the strip in the plane
  \begin{align}
     \mathbf{S}_3 = \{ k= \alpha+\im \beta  \text{ such that  }   -\sqrt{2}  \le   \beta   \le 1  \}  , \label{eq:f34}
\end{align}
   there exists a solution $f_3(x,k)$ of  \small
  \begin{align}\label{eq:f31}
    H f_3(x,k) =( k^2+1) f_3(x,k)  \text{ with } (x,k)\in  \R \times \mathbf{S}_3,
  \end{align}\normalsize
  with $  m_3(x,k) := e ^{x\sqrt{2+k^2}}f_3(x,k)$   satisfying  for any given $x_0\in \R$
    estimates
  \begin{align}\label{eq:f32} &
     |  m_3(x,k) -e_2     |\lesssim   \< k\> ^{-1 } e^{-\gamma x}  \text{ for $x\ge x_0$    and  } \\&  | \partial _x m_3(x,k)       |\lesssim   \< k\> ^{-1 } e^{-\gamma x} . \label{eq:f32b}
  \end{align}
  Furthermore  the map $ (p, k)\to \partial _x ^{\ell }m_3(\cdot ,k)$ is in $C ^{\omega} \( (p_1, +\infty ) \times    \mathring {\mathbf{S}  }_3 , L^\infty \(        x_0, +\infty \)   \)$ for any $x_0\in \R$ and for $\ell =0,1$, where $\mathring {\mathbf{S}  }_3$   is the interior of ${\mathbf{S}  }_3$.

\end{lemma}
\proof  Like in \cite{DT} or more specifically \cite[Lemma 5.2]{KrSch} we have the Volterra equation
\begin{align}& \label{eq:f33}
  m_3(x,k) =e_2  +\int _x ^{+\infty}  K(k,x-y)
     V(y)  m_3(y,k) dy  \text{ with }\\& K(k,x-y)=
   \diag \(     \frac{\sin \( k (y-x)\) }{k}e^{  \sqrt{2+k^2}  (x-y)} ,  \frac{e^{  2\sqrt{2+k^2}  (x-y)} - 1 }{\sqrt{2+k^2}} \) . \nonumber
\end{align}
We have
\begin{align*}
  \left |   \frac{e^{  2\sqrt{2+k^2}  (x-y)} - 1 }{\sqrt{2+k^2}}   \right | \lesssim  \< k \> ^{-1}  \text{ for } y\ge x.
\end{align*}
For $k$ not too close to 0  for $  -\sqrt{2}  \le   \beta   \le 1  $    by Lemma  \ref{lem:harm},  we have  \small
\begin{align*}
    \left |   \frac{\sin \( k (y-x)\) }{k}e^{  \sqrt{2+k^2}  (x-y)}       \right | \lesssim |k| ^{-1}   \sum _{\pm}    e^{\(  \pm \beta   -\Re \sqrt{2+ \alpha^2-\beta^2+2\im \alpha\beta}    \)   (y-x) }    \le   |k| ^{-1}  .
\end{align*}
\normalsize
For $k$   close to 0 writing
 \begin{align}
   \label{eq:f34}  \frac{\sin \( k (y-x)\) }{k} = 2^{-1} \int _{x-y} ^{y-x}  e^{\im k t }  dt
 \end{align}
  we get a bound
\begin{align*}
    \left |   \frac{\sin \( k (y-x)\) }{k}e^{  \sqrt{2+k^2}  (x-y)}        \right | \lesssim      |y-x|.
\end{align*}
 We can solve \eqref{eq:f33} recursively like in \cite{DT}  by
 $ m_3(x,k) = \sum _{n=0}^{\infty}m _{3,n}(x,k)$  with
\begin{align} m_3(x,k)& = \sum _{n=0}^{\infty}m _{3,n}(x,k)  \text{ where }   m _{3,0}(x,k) =e_2 \text{ and}\label{eq:series}
  \\ m _{3,n}(x,k)&=  \int _x ^{+\infty}  K(k,x-y)
     V(y) m _{3,n-1}(y,k) dy .\nonumber
    \end{align}
Then like in  \cite{DT}    for  $k$ not too close to 0 we get the following which yields    \eqref{eq:f32} by $|V(t)|\lesssim e^{-(p-1)|t|}$,
\begin{align*}
  | m _{3,n}(x,k)|&\le  |k| ^{-n} C_{x_0} ^n \int _{x\le x_1\le ...\le x_n} dx_1...dx_n      | V(x_1) | ...      |V(x_n) |
  \\& =  |k| ^{-n} C_{x_0} ^n \frac{\( \int _x ^{+\infty}       |  V(t) | dt\) ^n  }{n!}.
\end{align*}
  For $k$   close to 0  we have instead    \small
  \begin{align*}
  | m _{3,n}(x,k)|&\le    C_{x_0} ^n \int _{x\le x_1\le ...\le x_n} dx_1...dx_n (x_1-x)...(x_n- x_{n-1 } )            | V(x_1) | ...     | V(x_n) |
  \\& =   C_{x_0} ^n \frac{\( \int _x ^{+\infty}  (t-x)     |         V(t) | dt \) ^n }{n!}.
\end{align*}
    \normalsize
  From this we   get   \eqref{eq:f32} also for  $k$  close to 0.
  Finally the last statement is true for all the $(p,k)\to m _{3,n}(\cdot ,k)$ and by the local uniform convergence of the series
  \eqref{eq:series}. We skip the proof of the results on $\partial _x  m _{3 }(x ,k)$  which is similar and standard.

    \qed

\begin{lemma}
  \label{lem:f1} Consider the notation and hypotheses in Lemma \ref{lem:f3} and let us pick a small  $\varepsilon _0>0$. Then
   there exists a solution $f_1(x,k)$ of  \small
  \begin{align}\label{eq:f11}
    H f_1(x,k) =( k^2+1) f_1(x,k)  \text{ with } (x,k)\in  \R \times    \mathbf{S} _1   \text{ for }    \mathbf{S} _1:=\mathbf{S} _3\cap \{ \alpha+\im \beta : 1-\varepsilon_0\ge \beta \ge  -\gamma /2 \}   ,
  \end{align}\normalsize
  with $  m_1(x,k) := e ^{-\im k x }f_1(x,k)$   satisfying  for any given   $x_0\in \R$
  the estimate, for some $\varepsilon _1>0$
  \begin{align}\label{eq:f12}
     |  \partial _{x} ^{\ell }(m_1(x,k) -e_1)     |\lesssim   \< k\> ^{-1 +\ell   } e^{-\varepsilon _1 x}  \text{ for $x\ge x_0$  for $\ell =0,1$.}
  \end{align}
  Furthermore  the map $ (p, k)\to \partial _{x} ^{\ell }m_1(\cdot ,k)$ is in $C ^{\omega} \( (p_1, +\infty ) \times  \mathring {\mathbf{S}  } _1  , L^\infty \(        x_0, +\infty \)   \)$  for any $x_0\in \R$  and $\ell =0,1$.

\end{lemma}
\proof Following Krieger and Schlag \cite[Lemma 5.3]{KrSch}    we can frame the above as
\begin{align}\label{eq:f14}
  f_1(x,k) =  v(x,k) e_1 +   u(x,k) f_3(x,k)
\end{align}
where, for $ f_3  (x,k)  =\( f_3 ^{(1)}(x,k), f_3 ^{(2)}(x,k)\) ^\intercal $,
\begin{align}\label{eq:f13}
   u' (x,k) = \left [  f_3 ^{(2)}(x,k)   \right ] ^{-1} \int _{x}^{+\infty}  f_3 ^{(2)}(y,k)V _{21}(y) v(y,k) dy
\end{align}
and where  $v(x,k)= e^{\im kx}m(x,k) $ with
\small
\begin{align}\label{eq:votf1}
    m (x,k) &= 1 + \int _{x}^{+\infty} K(x,y,k) m(y,k) dy \text{  with  } K(x,y,k) = K_1(x,y,k) + K_2(x,y,k) \\ K_1(x,y,k) &= \frac{ e ^{2\im k (y-x)} -1 }{2\im k} \( \left [  f_3 ^{(2)}(y,k)   \right ] ^{-1}    f_3 ^{(1)}(y,k)  V _{21}(y)    + V _{11}(y) \) \text{ and}\nonumber \\  K_2(x,y,k) &= -2 \int_{x}^{y} \frac{ e ^{2\im k (z-x)} -1 }{2\im k}   \(   f_3 ^{(1)}(z,k) '-  \frac{f_3 ^{(2)}(z,k) '}{f_3 ^{(2)}(z,k)}    f_3 ^{(1)}(z,k)     \)  \left [  f_3 ^{(2)}(z,k)   \right ] ^{-2} dz  f_3 ^{(2)}(y,k)V _{21}(y) .\nonumber
\end{align}
\normalsize
Notice here  that   $V _{ij}(x) = M _{0ij}\sech ^2   \(   \frac{p-1}2 x\)  $ with $M _{0ij}$ components of the  matrix $M_0$, see   \eqref{eq:opH}.  Then, using Lemma \ref{lem:f3}  like in Krieger and Schlag \cite{KrSch}  we get for $y\ge x \ge x_0\ge 0$ and $k\in \mathbf{S}_1$
\begin{align*}
   |K  (x,y,k) |&  \lesssim  \frac{y-x}{\< |k| ( y-x) \>} e^{-\( p-1  -\gamma   \) y}   .
\end{align*}
This yields  a solution $m(x,k)$ of the Volterra equation \eqref{eq:votf1} with $| m(x,k) -1 |\lesssim \< k \> ^{-1}e^{-\( p-1    - \gamma \) x}$.  Inserting this in \eqref{eq:f13} for $k\in \mathbf{S}$  we get
\begin{align*}
 | u' (x,k) |\lesssim   \< k \> ^{-1}   \int _{x}^{+\infty}  e^{-\Re \sqrt{2+ \alpha^2-\beta^2+2\im \alpha\beta}  (y-x)-\( (p-1) + \beta  \) y  }       dy \lesssim  \< k \> ^{-1} e^{-\( (p-1) -  \frac{\gamma}{2} \) x} .
\end{align*}
Then, choosing $  u  (x_0,k)=0$ we have
\begin{align*}
   |u  (x ,k)|\lesssim  \< k \> ^{-1} .
\end{align*}
Finally, entering the above information in \eqref{eq:f14}  we get that for $m_1(x,k) = e^{-\im kx }f_1(x,k)$ we have
\begin{align*}
  |m_1(x,k) -e_1|&\lesssim  |m (x,k) - 1| +   \< k \> ^{-1} |m_3(x,k)| e^{ -\(  \Re \sqrt{2+ \alpha^2-\beta^2+2\im \alpha\beta  }  +\beta \)   x}\\& \lesssim  \< k \> ^{-1}e^{-\( p-1    - \gamma \) x}  + \< k \> ^{-1} |m_3(x,k)| e^{ -\(  \Re \sqrt{2+ \alpha^2-\beta^2+2\im \alpha\beta  }  +\beta \)   x} .
\end{align*}
Notice from the proof of Lemma \ref{lem:harm}  for  the last exponential for $k \in  \mathbf{S}_1$  we have for some $\varepsilon _1>0$
\begin{align*}
   e^{ -\(  \Re \sqrt{2+ \alpha^2-\beta^2+2\im \alpha\beta  }  +\beta \)   x} \le  e^{ - \varepsilon _1     x}  .
\end{align*}
Finally, the last statement follows because   $m(\cdot ,k)$ (like $m_3$) is in  $C ^{\omega} \( (p_1, +\infty ) \times \mathring {\mathbf{S}  } _1 , L^\infty \(        x_0, +\infty \)   \)$. As a consequence also       $u'(\cdot ,k)$,   $u (\cdot ,k)$ and  $m_1(\cdot ,k)$ satisfy the last statement. We skip the proof of the results on $\partial _x  m _{1 }(x ,k)$  which is similar and standard.

  \qed

By taking $ f_2(x,k):=  \overline{f_1(x,\overline{k})} $     we define another Jost function corresponding to the analogous function in \cite{KrSch}.

\begin{lemma}
  \label{lem:f4} Consider the notation and hypotheses in Lemma \ref{lem:f3}.
 \begin{description}
   \item[a] There exists a solution $\widetilde{f}_4(x,k)$   \small
  \begin{align}\label{eq:f41}
    H \widetilde{f}_4(x,k) =( k^2+1) \widetilde{f}_4(x,k)  \text{ with } (x,k)\in  \R \times   \widetilde{\mathbf{S}} _4  \text{ for }   \widetilde{ \mathbf{S}} _4=  \{ \alpha+\im \beta : |\beta |\le 1  \}   ,
  \end{align}\normalsize
  with $  \widetilde{m}_4(x,k) := e ^{-   \sqrt{2+k^2} x }\widetilde{f}_4(x,k)$   satisfying  for any  given  $x_0\in \R$
    estimate
  \begin{align}\label{eq:f42}
     | \partial _x ^{\ell }   ( \widetilde{m}_4(x,k) -e_2)     |\lesssim   \< k\> ^{-1 +\ell } e^{-\gamma  x}  \text{ for $x\ge x_0$  and $\ell =0,1$.}
  \end{align}
   \item[b]  The map $ (p, k)\to \partial _x ^{\ell } \widetilde{m}_4(\cdot ,k)$ is in $C ^{\omega} \( (p_1, +\infty ) \times  \mathring {\widetilde{\mathbf{S}}  } _4 , L^\infty \(        x_0, +\infty \)   \)$  for $\ell =0,1$.
   \item[c]  There is a unique choice of $c_1,c_2\in \C$ such that
  \begin{align}\label{eq:jost14} {f}_4 (x,k) :=-c_1  f_1 (x,k)  -c_2 f_2 (x,k)+ \widetilde{f}_4 \Longrightarrow     W[f_j,f_4]=0 \text{ for $j=1,2$.} \end{align}

  \item[d] For  $   {m}_4(x,k) := e ^{-   \sqrt{2+k^2} x } {f}_4(x,k)$ we have for $\mathbf{S}_4=  \{ \alpha+\im \beta : |\beta |\le a_2  \}$,  $a_2:=\min \( 1-\varepsilon _0, \frac{\gamma}{2}\)$
  \begin{align}
   \text{$(p, k)\to  \partial _x ^{\ell } {m}_4(\cdot ,k)$ is in $C ^{\omega} \( (p_1, +\infty ) \times  \mathbf{S} _4 , L^\infty \(        x_0, +\infty \)   \) $ for $\ell =0,1$.} \label{eq:f42b}
  \end{align}

 \end{description}
  \end{lemma}
\proof  We consider the equation $  \widetilde{m}_4(\cdot ,k) = e_2 + T \widetilde{m}_4(\cdot ,k)$    like in \cite{KrSch},  where
\begin{align}  \label{eq:f43}
Tf(x) &:= -\int _x ^{+\infty}
    \diag \(    0 ,  \frac {  1 }{\sqrt{2+k^2}} \)    V(y)   f(y)  dy \\& + \int  _{x_0} ^{x}
   \diag \(     \frac{\sin \( k (y-x)\) }{k}e^{  -\sqrt{2+k^2}  (x-y)} ,  \frac{ e^{  -2\sqrt{2+k^2}  (x-y)}   }{2\sqrt{2+k^2}} \)     V(y)   f(y)  dy   . \nonumber
\end{align}
By Lemma \ref{lem:harm} for $z\ge 0$    and  $ |  \beta |\le 1$ we have
\begin{align}
 \left | \sin \( k z\)   e^{  -\sqrt{2+k^2}  z}         \right |  \le e^{\(  \beta  -\Re    \sqrt{2+ \alpha^2-\beta^2+2\im \alpha\beta  } \) z     }  + e^{\(  -\beta  -\Re    \sqrt{2+ \alpha^2-\beta^2+2\im \alpha\beta  } \)z     }\le 2   .\label{eq:f44}
\end{align}
For $k$ not close to 0   from  \eqref{eq:f43} for$ |  \beta |\le 1$ we obtain
  \small
\begin{align*}
   |T(f-g)(x)| & \lesssim  \< k\> ^{-1}\int _{x}^{+\infty }e^{-(p-1) y} |f(y)-g(y)| dy  \\&\le  (p-1)^{-1} \< k\> ^{-1}  e^{-(p-1) x_0} \| f-g\| _{L^\infty (x_0,+\infty )} \text{ for }x\ge x_0.
\end{align*}
\normalsize
For  $k$  close to 0  using \eqref{eq:f34}  we get
\begin{align*}
   |T(f-g)(x)| & \lesssim \int _{x}^{+\infty }e^{-(p-1) y} |f(y)-g(y)| dy   + \int  _{x_0} ^{x} y  e^{-(p-1) y}    |f(y)-g(y)| dy      \\&\le    C_{\gamma}    e^{-\gamma x_0} \| f-g\| _{L^\infty (x_0,+\infty )} \text{ for }x\ge x_0.
\end{align*}
Notice that while in the statement of our lemma we have generically $x_0\in \R$, in the proof we can take   $x_0\gg 1$   and  conclude that $T$ is a contraction. This yields the solution $\widetilde{m}_4(\cdot ,k) \in L^\infty (x_0,+\infty )$ and provides the proof of \eqref{eq:f42} which extends also to the case  of any $x_0\in \R$.
Obviously we have by Neumann expansion
\begin{align*}
   \widetilde{m}_4(\cdot ,k) = \sum _{n=0}^{+\infty}T^n e_2
\end{align*}
where the series converges in $L^\infty (x_0,+\infty )$   for $x_0\gg 1$. Since  it is elementary to see  that the series   converges  uniformly  locally in terms of the parameters $ (p, k)\in (p_1,+\infty ) \times \widetilde{{\mathbf{S}}}_4  $
we get claim  \textbf{b} in the    statement and the result extends to any $x_0\in \R$.

Notice that $W[ f_1(\cdot ,k),f_2(\cdot ,k)] =2\im k $   and
$W[ f_3(\cdot ,k),\widetilde{f}_4(\cdot ,k)] = -2 \sqrt{2+k^2} $, see    \cite{KrSch} and analytic continuation. This yields claims  \textbf{c}   and    \textbf{d} of the statement, where $\R \times \mathbf{S}_4$ is contained in the domain of all the above Jost functions. We skip the proofs for  $\ell =1$.

  \qed

 Since   $V(p, -x) =V(p, x)$, writing
  \begin{align}\label{eq:jost21} &
    g_j (x,k)  :=
    f_j (-x,k)
  \end{align}
yields analogous Jost functions with prescribed behavior as $x\to -\infty$  with analogous results to those valid for the $f_j ( x,k)$, except for the fact that the half lines $(x_0,+\infty)$ need to be replaced by the half lines $( -\infty ,x_0)$.

\noindent For $p=3$ the Jost functions have been known explicitly since Kaup \cite{kaup}. We computed them explicitly in \cite{CM24D1} by exploiting Martel's Darboux transformation, see \cite{Martel2}. Using the result in \S 10   \cite{CM24D1}  it is elementary to    obtain  formulas
\begin{align}\label{eq:plwave1}
   f_1(x,k)&= \frac{e^{\im kx }}{1-k^2 -2\im k} \begin{pmatrix} 1-k^2 -2\im k\tanh (x) - \sech ^2 (x)   \\   - \sech ^2 (x) \end{pmatrix} \text{  ,  }\\ \label{eq:plwave2}
   f_2(x,k)&= \frac{e^{-\im kx }}{1-k^2 +2\im k} \begin{pmatrix} 1-k^2 +2\im k\tanh (x) - \sech ^2 (x)   \\   - \sech ^2 (x) \end{pmatrix} \text{  ,  } \\   {f}_3(x,k)&=  \frac{e^{ -\sqrt{2+k^2} x}}{3+k^2   + 2\sqrt{2+k^2}} \begin{pmatrix}  - \sech ^2 (x)     \\  3+k^2   + 2\sqrt{2+k^2} \tanh (x) - \sech ^2 (x) \end{pmatrix}  \text{  ,  } \label{eq:plwave3}   \\   {f}_4(x,k)&=  \frac{e^{ \sqrt{2+k^2} x}}{3+k^2   - 2\sqrt{2+k^2}} \begin{pmatrix}  - \sech ^2 (x)     \\  3+k^2   - 2\sqrt{2+k^2} \tanh (x) - \sech ^2 (x) \end{pmatrix}  \text{  .  } \label{eq:plwave4}
\end{align}
We consider the matrices
\begin{align}\label{eq:matjost}&  F_1 (x,k) = (f_1 (x,k), f_3 (x,k) ) \quad , \quad  F_2 (x,k) = (f_2 (x,k), f_4 (x,k) )  \, , \\&
  G_1 (x,k) = (g_2 (x,k), g_4 (x,k) ) \quad , \quad  G_2 (x,k) = (g_1 (x,k), g_3 (x,k) )  .\nonumber
\end{align}
For   matrix valued functions     $F= (\phi _1 (x), \phi _2(x)) $   and   $G= (\psi _1(x), \psi _2(x)) $  (where $\phi _1(x), \phi _2(x), \psi _1 (x)$ and $\psi _2(x)$ are column vector functions),  we set
\begin{align*}
 W[F,G]:= F'(x)^ \intercal G(x) - F  (x)^ \intercal G '(x).
\end{align*}
By direct computation, see   \cite{KrSch},
\begin{align*}
 W[F,G] =    \begin{pmatrix} W[\phi_1,\psi _1]  &   W[\phi_1,\psi _2] \\ W[\phi_2,\psi _1]  &   W[\phi_2,\psi _2] \end{pmatrix}
 .
\end{align*}
Quoting from \cite{KrSch}  and using the above results we have the following.
\begin{lemma}\label{lem:trrefl} For any $k\in \mathbf{S} \backslash \{  0 \} $ there exist matrices $A(k) $ and $B(k)$, smooth in $k$ and  such that
   \begin{align}
     \label{eq:trrefl1}  F_1(x,k) =  G_1(x,k)  A(k) + G_2(x,k) B(k),
   \end{align}
   with $ A(-k)= \overline{A}(k)$, $ B(-k)= \overline{B}(k)$ and
   \begin{align} &  \label{eq:trrefl2}   G_2(x,k) =  F_2(x,k)  A(k) + F_1(x,k) B(k) ,    \\&    W[F_1(x,k),G_2(x,k)]  = A^ \intercal  (k)  \diag (2\im k, -2\sqrt{2+k^2}) \label{eq:trrefl3} \\&  W[F_1(x,k),G_1(x,k)]  = - B^ \intercal (k)    \diag (2\im k, -2\sqrt{2+k^2}) .\label{eq:trrefl4}
   \end{align}
   We have
   \begin{equation}\label{eq:complcon}
     \begin{aligned}
        & G_1(x,k)= F_2(-x,k)  \quad , \quad G_2(x,k)= F_1(-x,k) \\&  \overline{F_1( x,k)}= F_1( x,-k) \quad , \quad  \overline{F_2( x,k)}= F_2( x,-k).
     \end{aligned}
   \end{equation}
and furthermore  these are functions analytic in $p\ge 1 $ and $k\in \mathbf{S} _4$.

\end{lemma}  \qed

We set
\begin{align}\label{eq:defd}
  D(p, k):= W[ F_1(x,k) ,  G_2(x,k)] =  \begin{pmatrix} W[ f_1(x,k),g_1(x,k)]  &   W[ f_1(x,k),g_3(x,k)] \\ W[ f_3(x,k),g_1(x,k)]  &   W[f_3(x,k),g_3(x,k)] \end{pmatrix} .
\end{align}
Notice that $D\in C ^{\omega} \( (p_1, +\infty ) \times  \mathring { {\mathbf{S}}  }_1 , M(2,\C )   \)$ with  $ M(2,\C )$   the spare of  square 2 matrices   with coefficients in $\C$.

The following holds, see   \cite{KrSch}.

\begin{lemma} \label{lem:eig0}For $k\in \R \setminus \{  0 \}$ the following are equivalent: \begin{itemize}
                                                                            \item  $\det A(k)=0$;
                                                                            \item $E=1+k^2$ is an eigenvalue of $H$;
                                                                            \item  $\det D(p,k)=0$.
                                                                       \end{itemize}    Furthermore   $E=1$ is neither a resonance  nor an eigenvalue   of $H$ if and only if   $\det D(p,0)\neq 0$.

\end{lemma}
\qed

Following the arguments in Krieger and Schlag \cite{KrSch}  we obtain the following.
\begin{lemma}
  \label{lem:eig000} If $E=1+k^2$  for $ 0<\Im k < 1 $ is an eigenvalue   with $p>p_1$  then $\det D(p,k)=0$.
\end{lemma}
\proof Notice first of all that $D(p, k)$  is well defined in the set of the statement. If we   have an eigenvalue  with   $ 0<\Im k \le  \min \( 1-\varepsilon _0, \frac{\gamma}{2}\)$ then  the corresponding eigenfunction is of the form
\begin{align}\label{eq:eig0001}
  u(x)=  F_1(x,k) \alpha +F_2(x,k)\beta \text{     with column vectors $\alpha ,\beta  \in \C ^2$.}
\end{align}
  But since
$|F_1(x,k)  |  \xrightarrow{x \to  +\infty} 0$ while  $|F_2(x,k) \beta  |  \xrightarrow{x \to  +\infty} +\infty$  for any $\beta \neq 0$, we must have  $\beta =0$ and $\alpha \neq 0$.
  Applying the two sides of  \eqref{eq:trrefl1} to $\alpha  $ we have
\begin{align}\label{eq:eig0002}
  u(x)=  F_1(x,k) \alpha =  G_1(x,k)  A(k)\alpha  + G_2(x,k) B(k)\alpha .
\end{align}
If now $ A(k)\alpha \neq 0$  by $| G_1(x,k)  A(k)\alpha  | \xrightarrow{x \to -\infty} +\infty$ and  $|F_1(x,k) \alpha |+ | G_2(x,k)| \xrightarrow{x \to -\infty} 0$
we get a contradiction. So $ A(k)\alpha =0$ and hence
$\det A(k)=0$ which by \eqref{eq:trrefl3} and the definition   of $ D(p, k)$  in \eqref{eq:defd}   is equivalent to   $\det D(p,k)=0$.  So we have proved the statement for    $ 0<\Im k \le \min \( 1-\varepsilon _0, \frac{\gamma}{2}\)$, a range which guarantees    the existence of all the Jost functions.  The argument extends however to $ 0<\Im k < 1$. The reason is the following. First of all $\varepsilon _0>0$ can be taken as small as desired which was fixed   in order to get the inequality in \eqref{eq:f12} for some $\varepsilon _1>0$.  Furthermore, if $V$ were compactly supported we could pick
$a_0=  1-\varepsilon _0$ and by the arbitrariness of $\varepsilon _0>0$ we would reach the whole range $ 0<\Im k < 1 $. Let us consider now a sequence $V_n(\cdot ) \xrightarrow{n\to +\infty} V (\cdot ) $ of compactly supported potentials   with the convergence occurring in
 $L^\infty _{p-1-\varepsilon} $  for any $\varepsilon >0$. Then it is easy to see that $D_n(\cdot , \cdot ) \xrightarrow{n\to +\infty} D (\cdot , \cdot ) $ uniformly on compact sets in the domain defined by $ 0<\Im k < 1 $ is an eigenvalue   with $p>p_1$ and this yields the statement of the lemma.

\qed

In view of Lemma \ref{lem:eig0} we have
\begin{equation}\label{eq:deftildeF1}
  \widehat{ \mathbf{F}} \cap (p_1, +\infty )\subseteq \{   p>p_1  : \  \det D(p,0)=0  \}   .
\end{equation}
Notice that  $D(p,k)$ depends analytically in $p>p_1$ and   for $-a_0<\Im k<1$.  By elementary computations we get
\begin{align*}
   \det D(3,k)= -4\im k \frac{1-k^2+2\im k}{1-k^2-2\im k}  \sqrt{2+k^2} \ \frac{3+k^2-2\sqrt{2+k^2}}{3+k^2+2\sqrt{2+k^2}}.
\end{align*}
This shows that $k=0$ is a root of multiplicity 1  for $ \det D(3,k)$.  Then  $ \det D(p,\cdot )=0$
   has exactly a single root close to 0   and it has  multiplicity 1 for  $0<|p-3|\ll 1 $.
Since for $0<|p-3|\ll 1 $   there is an eigenvalue $ \lambda(p) \in (0,1)$ very close to 1,  indeed  by Coles and Gustafson \cite{coles}      it satisfies
\begin{equation}\label{eq:lambdanear3}
         \lambda(p) = 1 - \alpha _0 (p-3) ^4 +o\( (p-3) ^4  \) \text{ for a constant $\alpha _0>0$,} \end{equation}
           by   Lemma  \ref{lem:eig000}    there exists $k(p) \in  \im \R _+$ very close to 0 such that  $\lambda(p) =1+ k^2(p)$  with $ \det D(p, k (p))= 0$. So this is the single root close to 0
and hence  $ \det D(p, 0)\neq  0$  for
$0<|p-3|\ll 1 $.  Hence, by the analytic dependence of $ D(p,0)$ in $p$,  the set $ \{   p>p_1 : \ \det D(p,0)=0  \}$ is a discrete subset of $(p_1,+\infty) $. But here $p_1>1$ can be taken arbitrarily close to  1, so we conclude that  $\{   p> 1 : \  \det D(p,0)=0  \}$ is a discrete subset of $( 1,+\infty) $,
 proving Proposition \ref{prop:nonres}.

\qed

\section{The analyticity  of $\gamma _3 (p) $ and proof of Proposition \ref{prop:n3FGR}} \label{sec:andep}

In this section, we prove Proposition \ref{prop:n3FGR}.
The idea is to use analyticity and show that $\gamma_3(\cdot )$ is not identically zero. Notice that in this section we will always consider $p>3- \varepsilon _0$ for some small $\varepsilon _0$ and we will  not take $p$ close to 1 like in Section \ref{sec:prop:nonres}.
We decompose  $\gamma_3(\cdot )$ as
\begin{align*}
    \gamma_3(p)=\<G_{3,1}(p),g_3(p)\> + \<G_{3,2}(p),g_3(p)\>=: \gamma_{3,1}(p) +\gamma_{3,2}(p),
\end{align*}
where $G_{3,j}(p)=(G_{(3,0),j},G_{(0,3),j})^\intercal $ for $j=1,2$, $G_{(3,0),1}$ is given by the 1st line of the r.h.s.\ of \eqref{def:G30} and $G_{(3,0),2}$ is given by the 2nd and 3rd line of the r.h.s.\ of \eqref{def:G30}.
Further, $G_{(0,3),j}$ are given by replacing all $\xi_{\mathbf{m}}$ in \eqref{def:G30} by $\xi_{\overline{\mathbf{m}}}$.

The functions $\gamma_{3,1}$ and $\gamma_{3,2}$ are initially defined on $(p_2,p_3)$. Proposition  \ref{prop:n3FGR} will be a consequence of Lemmas \ref{lem:gamma31}-- \ref{lem:Gamma23}  below.
\begin{lemma}\label{lem:gamma31}
There exists a connected open set $\mathfrak{A}_1\subset \C$ and a function $\Gamma_1$ holomorphic in $\mathfrak{A}_1$ such that \ $(p_2,p_3)\subset \mathfrak{A}_1$, $3\in \overline{\mathfrak{A}_1}$ and $\Gamma_1|_{(p_2,p_3)}=\gamma_{3,1}$.
Further, $\Gamma_1$ can be continuously extended to $p=3$.
\end{lemma}

\begin{lemma}\label{lem:gamma32}
$\gamma_{3,2}$ can be analytically extended on $(3,p_3)$.
\end{lemma}


At $p=3$, we can compute the imaginary part of $\Gamma_1$ and $\Gamma_2$.
\begin{lemma}\label{lem:Gamma13}
    $ \displaystyle  \Im \Gamma_1(3):=\lim_{p\in \mathfrak{A}_1, p\to 3}\Im \Gamma_1(p)\neq 0$.
\end{lemma}

\begin{lemma}\label{lem:Gamma23}
	$\displaystyle \lim_{p\to 3+}\Im \gamma_{3,2}(p)=0$.
\end{lemma}

\begin{proof}
The function $\gamma_{3,2}$ is real valued at $(p_2,p_3)$.
Thus, its analytic extension on $(3,p_3)$, given in Lemma \ref{lem:gamma31}, is also real valued.
\end{proof}

Before proving lemmas \ref{lem:gamma31}, \ref{lem:gamma32} and \ref{lem:Gamma13}, we prove Proposition \ref{prop:n3FGR}.
\begin{proof}[Proof of Proposition \ref{prop:n3FGR}]
Since $\gamma_{3,2}$ is analytic on $(3,p_3)$, we can analytically extend it on an open set $\mathfrak{A}_2\subset \C$.
Then, from Lemmas \ref{lem:gamma31} and \ref{lem:gamma32}, we define $\Gamma=\Gamma_1+\gamma_{3,2}$ in $\mathfrak{A}=\mathfrak{A}_1\cap \mathfrak{A}_2$.
Then, $\Gamma$ is holomorphic in $\mathfrak{A}$, $\Gamma|_{(p_2,p_3)}=\gamma_3$ (notice that $(p_2,p_3)\subset \mathfrak{A}$) and $\Im \Gamma$ can be continuously extended to $p=3\in \overline{\mathfrak{A}}$ with $\lim_{p\in \mathfrak{A},\ p\to 3}\Im \Gamma(p)\neq 0$.
Thus, $\Gamma$ is a nonzero holomorphic function in $\mathfrak{A}$ and therefore has at most a discrete  zero sets in $(p_2,p_3)$.
From $\Gamma|_{(p_2,p_3)}=\gamma_3$, we have the conclusion.
 \end{proof}

The proof of Lemmas \ref{lem:gamma31}, \ref{lem:gamma32} consists of carefully studying the analytic extension of $\phi_p$, $g_3(p)$, $\xi_{\mathbf{ m}}(p)$, which are the building blocks of $\gamma_{3,1}$ and $\gamma_{3,2}$.
For the proof of Lemma \ref{lem:Gamma13} we compute explicitly the above building blocks, which is possible due to the good factorization property of the linearized operator $\mathcal{L}$ at $p=3$.

We start from the analyticity of the eigenvalue $\lambda(p)$.

 \begin{lemma} \label{lem:analambda} There exists a small $\varepsilon _0>0$ such that the map $ p\to \lambda (p)$ is   in $ C^\omega  \( (3-\varepsilon _0, 5), \C \) $.
    \end{lemma}
 \proof
  From \eqref{eq:opH} we see that $H$ is analytic of type $(A)$  in $p>1$  and this implies that the simple eigenvalue $\lambda (p)$ depends analytically on $p \in (3, 5)\cup (3-\varepsilon _0,3)$, see \cite[Ch. 12]{reedsimon}. Notice here that for $p\in (3, 5)$ away from $3$ we are relying on Assumption \ref{ass:eigenvalues1} while for    $p\in (3-\varepsilon _0,3 )$  and  $p\in (3, 3+\varepsilon   )$ on the facts proved by Coles and Gustafson \cite{coles}, see also \cite[Lemma 11.3]{CM24D1}.
 In fact, inspection of the proof  in \cite[Lemma 11.3]{CM24D1} shows that the above extends to the statement  that  $ p\to \lambda (p)$ is analytic in $p \in (3-\varepsilon _0, 5)$. Indeed $\lambda (p)=1-\alpha ^2(p)$
 where, for $p$ near 3, $\alpha (p)$ is the implicit function of the equation
 \begin{align}\label{eq:implalpha}
	\alpha  =-\frac{p-3}{2}
	s(p,\alpha )   \text{  with }   s(p,\alpha ):= \<   |\mathbf{P}_p| ^{\frac{1}{2}}(1+(p-3) M _{\alpha p} ) ^{-1}  \mathbf{P}_p^{\frac{1}{2}}   e_2   , e_2 \> _{\C ^2} .
\end{align}
 where here for    the $\phi_p (x)$  in  \eqref{eq:sol}   we have
 \begin{equation}\label{eq:implalpha1}
    \begin{aligned}
      &|\mathbf{P}_p(x) | ^{\frac{1}{2}}  :=   \begin{pmatrix}
		1+\sqrt{p-2} & 1-\sqrt{p-2}\\
		1-\sqrt{p-2} & 1+\sqrt{p-2}
	\end{pmatrix}      \frac{1}{2 \sqrt{p+1} }\phi_p  ^{\frac{p-1}{2}} (x)  \ ,  \\ &\mathbf{P}_p ^{\frac{1}{2}}(x)  :=    \sigma _1  |\mathbf{P}_p(x) | ^{\frac{1}{2}} =   \begin{pmatrix}
		1-\sqrt{p-2} & 1+\sqrt{p-2}\\
		1+\sqrt{p-2} & 1-\sqrt{p-2}
	\end{pmatrix}      \frac{1}{2 \sqrt{p+1} }\phi_p  ^{\frac{p-1}{2}} (x)\  , \\&  M _{\alpha p}  := \mathbf{P}_p^{\frac{1}{2}} N _{\alpha  }|\mathbf{P}_p| ^{\frac{1}{2}} \text{ and }\\&  N _{\alpha  }(x,y)= \frac{1}{2\alpha}    \begin{pmatrix} \frac{\alpha}{\sqrt{2- \alpha^2}} e^{-\sqrt{2- \alpha^2} |x-y|} &   0  \\  0     &   e^{-\alpha |x-y|} -1 \end{pmatrix}
    \end{aligned}
 \end{equation}
 with in the last line   the integral kernel of the operator $ N _{\alpha  }$. Since the function in \eqref{eq:implalpha}  is analytic in $(p,\alpha )$  for $p\in (1,5)$, and $0<\Re \alpha < \sqrt{2}$  it follows that the implicit function  $\alpha (p)$ is analytic for    $p$ in a neighborhood of  3, see     Dieudonn\'e  \cite[p. 272]{dieu}, and therefore $\alpha (\cdot )\in C^\omega ((3-\varepsilon _0, 5), \R ) $.
 \qed


 We next discuss the $g_3$ given in \eqref{eq:eqsatg2}. We will use the spaces \eqref{L2g}.

   \begin{lemma}\label{lem:analyg}
  There exists a   $\varepsilon _0 >0$ and a $\R^2$-valued function
  \begin{align}\label{eq:analyg3}
    g_3 \in C^\omega \( (3- \varepsilon _0, p_3),  L^\infty  _{-s}(\R ) \)
  \end{align}
    for any $s>0$  
    such that $Hg_3(p) =3\lambda (p) g_3(p)$    for any $p$  and
\begin{align}\label{eq:analyg1}
    g_3(3)=\begin{pmatrix}
    \cos(\sqrt{2}x)  + \mathrm{sech}^2(x) \cos(\sqrt{2}x) - 2\sqrt{2} \tanh(x) \sin(\sqrt{2}x) \\
    \sech^2(x) \cos(\sqrt{2}x)
    \end{pmatrix}.
\end{align}
 \end{lemma}
 \begin{proof}
Notice that using the function  $f_1(x,k)$ in    \eqref{eq:plwave1} we have
\begin{align*}
   g_3(3)=- \Re \left [ \left . \(  1-k^2 -2\im k \) f_1(x,k) \right | _{k=\sqrt{2}} \right ] .
\end{align*}
By Lemma \ref{lem:f1} the function    $f_1(x,k)$    is analytic in $(k,p)$  for $p>3-\varepsilon _0$  and $ k= \alpha+\im \beta $ with $-\( 1 -\frac{\varepsilon_0}{2}\)\le \beta \le 1 - {\varepsilon_0}$ where we can take as $\varepsilon _0>0$ the same number of Lemma \ref{lem:f1}. By  Lemma  \ref{lem:f3}   the function    $f_1(x,k)$    is analytic in $(k,p)$ in the same set. So the same holds for
    $ F_1 (x,k) = (f_1 (x,k), f_3 (x,k) )$.      Now recall
that by  Krieger and Schlag \cite[Lemma 6.3]{KrSch}, for $e_1$ like in \eqref
 {eq:versors}  and in the notation in \S \ref{sec:prop:nonres} , we have     $\widetilde{g}_3(p)(\cdot )\in L^\infty (\R ) $ for
\begin{align}\label{eq:analyg2}
 \widetilde{ g}_3(p)(\cdot ):= 2\im \kappa  F_1(\cdot , \kappa ) D^{-1}(p,\kappa ) e_1  \text{ with }   \kappa= \kappa(p) =\sqrt{9\lambda ^2(p)-1} .
\end{align}
It is easy to show that $  \widetilde{g}_3 (\cdot )\in  C^0\(  (3- \varepsilon _0, p_3),   L^\infty (\R ) \)$.
Notice that by Lemmas \ref{lem:f3} and \ref{lem:f1} for any $s>0$ we have $F_1(\cdot , \kappa \( \cdot )  ) \in C ^{\omega} (3- \varepsilon _0, p_3),   L^\infty _{-s} (x_0,+\infty ) \)$ for any  $x_0\in \R$.
From  \cite[Lemma 6.3]{KrSch}  we see also that
\begin{align*}
   \widetilde{ g}_3(p)(x )&=    f_2(-x, \kappa(p) ) +  G_2 (x, \kappa(p) ) B(\kappa(p))A^{-1}(\kappa(p)) e_1
\end{align*}
where the right hand side is in  $  C ^{\omega} \( (3- \varepsilon _0, p_3),   L^\infty _{-s} (-\infty , x_0 ) \)$ for any  $x_0\in \R$, see \eqref{eq:jost21} and below. Then we can glue the statements and conclude that  $\widetilde{g}_3 (\cdot )\in C^\omega \( (3- \varepsilon _0, p_3),  L^\infty  _{-s}(\R ) \)$.

\noindent   Since  a direct simple inspection shows that $ D (3, k)$  is diagonal,    formula  \eqref{eq:analyg2}
yields
for $p=3$ a multiple of the  Jost function in  \eqref{eq:plwave1}. Multiplying by an appropriate constant and taking the real part  we get the desired function    $ { g}_3  (p)$.
\end{proof}

We omit the simple proofs of the following two lemmas, which can be proved  using formula  \eqref{eq:sol}.
 \begin{lemma}\label{lem:phip-2}
There exists $\varepsilon_0>0$ such that  $\phi_p^{p-2}\in C^\omega((3-\varepsilon_0,5),L^2_{1/2})$
 \end{lemma}



\begin{lemma}\label{lem:phip-3}
For any $\varepsilon>0$, there exists $\delta>0$ such that  $\phi_p^{p-3}\in C^\omega((3+\varepsilon,p_3),L^2_\delta)$.
\end{lemma}

For $(\xi_{(1,0)},\xi_{(0,1)})^\intercal$, which is an eigenvector of $H$ associated to the eigenvalue $\lambda(p)$, we have the following.

\begin{lemma}\label{lem:analyxi01} There exists $\varepsilon _0>0$ such that
for any
  $s>0$   and for $m=(1,0) , (0,1)$  there exist real valued functions $\xi_{\mathbf{m}} \in C^\omega ((3-\varepsilon_0,p_3),L^\infty_{-s}(\R))$
   satisfying \eqref{eq:evlambda} and
\begin{align}\label{eq:xiat3}
\begin{pmatrix}
\xi_{(1,0)}(3)\\
\xi_{(0,1)}(3)
\end{pmatrix}
=\begin{pmatrix}
1-\sech^2(x)\\
-\sech^2(x)
\end{pmatrix}.
\end{align}
\end{lemma}

\begin{proof}
We know     that for   $p\in (3- \varepsilon _0, p_3)$ there is exactly one solution $k=k(p)$ in the segment $ \left [0, \im  \sqrt{\frac{2}{3}}  \) \subset \C$   to $\det D(p,k) =0$.
Furthermore we know that $D(p,k)$  is analytic in $p$ and $k$ and that $k(p)$    has multiplicity 1, that is $  \frac{\partial}{\partial k} \left . \det D(p,k)\right | _{k=k(p)}\neq 0$, so that
  the   map $p\to k(p)$ is in $C ^\omega ( (3- \varepsilon _0, p_3), \C )$. The projection  map
   \begin{align*}
     Q_p = -\frac{1}{2\pi \im } \oint _{\partial D_p}  \frac{dw}{D(p,k(p)) -w}  \in \mathcal{L}(\C ^2, \ker \( D(p,k(p)) - k(p)  \)   )
   \end{align*}
   over appropriate boundaries of disks $D_p$ containing   $k(p)$  in their interior  and the other eigenvalue of $D(p,k(p))$ in the interior of $\C \backslash D_p$. Then $p\to Q_p\in C ^{\omega}\(  (3- \varepsilon _0, 5) ,  \mathcal{ L}( \C ^2) \)$.   Let now $\alpha (3)$ be an eigenvector of $D(3,0)$ associated to  $k(3)=0$ and set $\alpha (p)=Q_p \alpha (3)$. We have
   $p\to \alpha (p) \in C ^{\omega}\(  (3- \varepsilon _0, p_3) ,   \C ^2  \)$   with $\alpha (p)=0$ at most in a discrete subset of $(3- \varepsilon _0, p_3)$ and with $\alpha (3)\neq 0$. Following   the discussion in Lemma \ref{lem:eig000} we write
     \begin{align*}
       u[p] (x) :=   F_1(x,k(p)) \alpha (p) .
     \end{align*}
     Then we have $(H-\lambda (p) )    u[p] =0$   with $ u[3] $ a multiple of the vector valued function in \eqref{eq:xiat3}, and in fact the same for the right choice of $\alpha (3)$.
     Proceeding line in Lemma \ref{lem:analyg} we know that   for any $s>0$ we have $F_1(\cdot , k ( \cdot )  ) \in C ^{\omega} \((3- \varepsilon _0, p_3),   L^\infty _{-s} (x_0,+\infty ) \)$ for any  $x_0\in \R$
     and so in particular we have $u[\cdot ] \in C ^{\omega} \( (3- \varepsilon _0, p_3),   L^\infty _{-s} (x_0,+\infty ) \)$.  To check the behavior in $(-\infty , x_0)$ we use formula \eqref{eq:eig0002}
     and since   here $\ker A(k) =\ker D(k) $  by \cite[formula (5.34)]{KrSch}, we have  by \cite[formula (5.27)]{KrSch}
     \begin{align*}
       u[p] (x)  =   G_2(x,k(p)) B(k(p))   \alpha (p)  \text{ where  } B(k) ^{\intercal}= W[F_1 (\cdot , k) ,G_1 (\cdot , k) ] \diag \(   \frac{\im }{2  k} ,  \frac{1}{2 \sqrt{2+k^2}}  \) .
     \end{align*}
     By choosing $\varepsilon _0>0$ sufficiently small, and taking the same $\varepsilon _0 $ also in the statement of Lemma \ref{lem:f1},  we obtain that $B (\cdot )\in C ^{\omega}
     \(  \{ k:  |\Im k |<   \sqrt{\frac{2}{3}} \}, M(2 , \C )  \)$. Then    $u[\cdot ] \in C ^{\omega} \( (3- \varepsilon _0, p_3),   L^\infty _{-s} (-\infty , x_0) \)$ for any  $x_0\in \R$
     and hence  $
      u[\cdot ] \in C ^{\omega} \( (3- \varepsilon _0, p_3),   L^\infty _{-s} \( \R \) \)
     $.
       Finally   choose
     $ \( \xi_{(1,0)}(p),
\xi_{(0,1)}(p)  \)^\intercal =   u[p]$.
\end{proof}

\begin{remark}\label{rem:kato} Notice that  we could have used also the \textit{transformation function }  in pp. 99--102 in Kato \cite{katobook} to define $\alpha (p) =F(p)   \alpha (3)   $   where
$F\in C ^{\omega}\(  (3- \varepsilon _0, 5) , GL( 2,\C  ) \)$  with $F(p) Q_{3}= Q_pF(p)$.
 \end{remark}
We can now prove Lemma \ref{lem:gamma32}.

\begin{proof}[Proof of Lemma \ref{lem:gamma32}]
Take $\varepsilon>0$ arbitrary and let $\delta>0$ the constant given in Lemma \ref{lem:phip-3}.
Then, since $\phi^{p-3}_p$ is analytic in the topology $L^2_\delta$ and $\xi_{(1,0)},\xi_{(0,1)}$ and $g_3$ are analytic in $(3+\varepsilon,p_3)$ in the topology $L^2_{-\delta/10}$, we have $\gamma_{3,2}$ is analytic in $(3+\epsilon,p_3)$.
Since $\varepsilon>0$ is arbitrary, we have the conclusion.
\end{proof}

From Lemma \ref{lem:analyxi01}, the analyticity of $G_{(2,0)}$ and $G_{(0,2)}$ immediately follows.
\begin{lemma}\label{lem:G20andG02analytic}
There exists $\varepsilon_0>0$ such that  for $\mathbf{m}=(2,0)$ and $(0,2)$, $G_{\mathbf{m}}\in C^\omega((3-\varepsilon_0,p_3),L^2_{1/4})$, where $G_{\mathbf{m}}$ are given in \eqref{def:G20} and below  \eqref{def:G11}.
\end{lemma}

\begin{proof}
Take $s=1/10$ in Lemma \ref{lem:analyxi01} and take $\varepsilon_0>0$ to be the smallest one appearing in Lemmas \ref{lem:phip-2} and \ref{lem:analyxi01}.
Then in $(3-\varepsilon_0,p_3)$,
$\phi^{p-2}_p$ is $L^\infty_{1/2}$-valued analytic function
and $\xi_{(1,0)}$ and $\xi_{(0,1)}$ are $L^2_{-1/10}$-valued analytic functions.
By the formula \eqref{def:G20}, we see that $G_{(2,0)}$ and $G_{(0,2)}$ are $L^2_{1/4}$-valued analytic functions.
\end{proof}

In the next lemma $(p_2, p_3)$  is an interval in the real axis.
\begin{lemma}\label{lem:Domdelta} For any $\theta _0\in \( 0, \frac{\pi}{6} \) $   and   $\delta   >0$ consider
\begin{align*}
    \mathfrak{C} _{\theta _0\delta  } :=\{p\in \C\ |\ 3\leq \Re p\leq p_3,\ -\delta<\Im p<0,\  0>\arg (p-3)>- \theta _0\}\cup (p_2,p_3) .
\end{align*}
 Then for $\delta    $   sufficiently small  we have $\Re \lambda ' (p)<0$   in $ \mathfrak{C} _{\theta _0\delta  }$  and $\Im \lambda   (p) >0$   in  $  \mathring {\mathfrak{C} }_{\theta _0\delta}$, the interior of $ \mathfrak{C}_{\theta _0\delta}$.
    \end{lemma}
\begin{proof} For  $p\in (3,5)$, we have $\lambda'(p)<0$ by Assumption \ref{ass:eigenvalues1}. So if we pick
$\widehat{p} \in ( 3, p_2)$ then if $\delta >0$ is sufficiently small,  by continuity   we have
$ \Re \lambda ' (p)<0 $  in the rectangle $[ \widehat{p} ,    p_2] \times (0, -\delta ) $.
Near $p=3 $ by \eqref{eq:lambdanear3}, for polar representation \eqref{eq:lambdanear3}    we have
\begin{align*}
    \Re \lambda'(p) =  -4 r^3 \cos{3\vartheta} +O\( r^4\) <-4r^3  \cos{3\theta _0}  \( 1+ O\( r \)  \)
\end{align*}
So if  $\widehat{r}=\widehat{p}-3$  is small enough such that
$\( 1+ O\( \widehat{r} \)  \) > \frac{1}{2} $
then for $|p-p_3|\le  \widehat{r}$    and in case taking $\delta$ smaller we get  $\Re \lambda ' (p)<0$
in $  \mathring {\mathfrak{C} }_{\delta}$.

\noindent For $p\in (3,5)$   for sufficiently small $\epsilon _p$, we have $\lambda(p-\im \epsilon_p)=\lambda(p) - \lambda'(p) \im \epsilon _p+O(\epsilon^2 _p)$, which implies $\Im \lambda(p-\im \epsilon_p)>0$.   By \eqref{eq:lambdanear3} again using \eqref{eq:lambdanear3} for $p =e ^{-\im \theta _0}   $ near 3
\begin{align*}
    \Im \lambda (p) =  \alpha _0 r^4  \sin\( 4 \theta  _0 \)  \( 1+ O(r) \) .
\end{align*}
So again for $r>0$ small enough  and by $0<4\theta_0<\pi $ we have $ \Im \lambda (p) >0$.  Taking $\delta >0$ small enough we conclude by the maximum principle that we have also $\Im \lambda   (p) >0$   in  $  \mathring {\mathfrak{C} }_{\theta _0\delta}$.

\end{proof}


We fix $ \theta _0 =\frac\pi  7$  and write $\mathfrak{C} _{ \delta  }  =\mathfrak{C} _{\theta _0\delta  }$.

\begin{lemma}\label{lem:xi2002extend}
There exist $\delta>0$ and a discrete set $\mathfrak{D}_\delta\subset \mathfrak{C}_{\delta}$ such that  there exists an open set $\mathfrak{A}_{9}$ satisfying $\mathfrak{C}_\delta\setminus \mathfrak{D}_\delta\subset \mathfrak{A}_{9}$ such that  $(\xi_{(2,0)} , \xi_{(0,2)})^ \intercal$ initially given by \eqref{def:xi20} on $(p_2,p_3)$ extends into an element in $C ^\omega \( \mathfrak{A}_9 , L^2(  \R  , \C ^2) \)$.
\end{lemma}

\begin{proof}
By Lemmas \ref{lem:analambda} and \ref{lem:G20andG02analytic},   $\lambda(p)$ and $(G_{(2,0)}(p)\ -G_{(0,2)}(p))^\intercal$ are analytic in $p$. So is also
the potential $V$ of $H_p$, see \eqref{eq:opH}.
Therefore, it follows that
\begin{align*}
  (\xi_{(2,0)},\xi_{(0,2)})^\intercal=(H_p-2\lambda(p))^{-1}(G_{(2,0)}(p) ,  -G_{(0,2)})^\intercal
\end{align*}
   is analytic in $\mathfrak{C}_\delta\setminus \mathfrak{D}_\delta$ where
$\mathfrak{D}_{\delta}:=\{p\in \mathfrak{C}_{\delta}\ |\ 2\lambda(p)\in \sigma_{\mathrm{d}}(H_p)\}$.
Thus, it suffices to show that $\mathfrak{D}_\delta$ is discrete.

\noindent Set $z=2\lambda(p)$.
By  $\lambda'(p)\neq 0$ in $p\in \mathfrak{C}_\delta$ and the fact that $\mathfrak{C}_\delta$  is simply connected,        we can take the inverse and write $p=p(z)$ with $p(\cdot)$   analytic.
Now suppose $z_0=2\lambda(p(z_0))$ is an eigenvalue of $H_{p(z_0)}$. Then, since we know that
$2\lambda (p)   $ is not an eigenvalue of $ H_{p }$ for $p_2<p<p_3$ we must have $p(z_0)\in \mathring {\mathfrak{C} }_{\delta}$, the interior of $ \mathfrak{C}_{\delta}$. By Lemma  \ref{lem:Domdelta}
$\Im \lambda(p(z_0))>0$ and hence $z_0\in \sigma_{\mathrm{d}}(H_{p(z_0 )}) $.

\noindent We show that if $z\neq z_0$ in a small neighborhood $z_0$, then $z\notin
\sigma_{\mathrm{d}}(H_{p(z )}) $.
Suppose $z_0$ is an eigenvalue of $H_{p(z_0)}$ with multiplicity $m\geq 1$.
For an appropriate closed path around $z_0$,
\begin{align*}
    P (z) :=  -\frac{1}{2\pi \im}  \int _\gamma R _{H_{p(z )}} (w) dw
\end{align*}
the eigenvalues of $H_{p(z )}$ near $z_0$ are also eigenvalues also of $H_{p(z )}P(z)$, which has finite constant rank for $z$ close to $z_0$. It is easy to conclude that   the study of the eigenvalues of $H_{p(z)}$ close to $z_0$   with $z$ are near $z_0$  reduced to the study of the eigenvalues of a matrix dependent analytically in $z$.
By   Corollary of Theorem XII.2 in \cite{reedsimon}   we  know
that there are  $k(\leq m)$ multi-valued analytic functions $\mu_1(z), \cdots, \mu_k(z)$ such that $\mu_j(z_0)=z_0$ and that the eigenvalues of $H_{p(z)}$ near $z_0$ are given by $\mu_1(z), \cdots, \mu_k(z)$ for $z$ near $z_0$ with an expansion
\begin{align*}
    \mu_j(z)=z_0+\sum_{l=1}^\infty \mu_l^k (z-z_0)^{l/p_j}  \text{ where $\displaystyle\sum_{j=1}^k p_j=m$.}
\end{align*}
Assume that $z_0$ is an accumulation point  of $\mathfrak{D}_{\delta}$. Then $z_0$ is an accumulation point of the set of solutions
  $\mu_j(z)=z$  for some $j$.
This implies   $\mu_j(z) = z$ identically for all $z$ near $z_0$, which implies that $2\lambda(p)$ is an eigenvalue of $H_p$ for $p$ in an open subset of $ \mathfrak{C}_{\delta}$.   This would imply that this is so for   all $p\in \mathring {\mathfrak{C} }_{\delta}$. But this would extend also to a larger set and in particular
  to $p\in (p_2,p_3)$ contradicting   Assumption \ref{ass:eigenvalues1}.
\end{proof}

The continuous extension, in a weaker topology, of $(\xi_{(2,0)}\ \xi_{(0,2)})$ to $p=3$ will be discussed in Lemma \ref{lem:zeta12}

\begin{lemma}\label{lem:fp}
Take $\varepsilon_0>0$ small so that we can analytically extend $G_{(2,0)}$ and $G_{(0,2)}$ for  $p\in D_\C(3,\varepsilon_0) $.
For $p\in D_\C(3,\varepsilon_0)$, set $\psi(p):=L_{+,p}^{-1}\phi_p$ (the inverse is taken in the space of even functions) and
\begin{align*}
    \begin{pmatrix}
    \im F_1(p)\\
    F_2(p)
    \end{pmatrix}
    =
    \im U
    \begin{pmatrix}
   G_{(2,0)}(p)\\ -G_{(0,2)}(p)
    \end{pmatrix},
\end{align*}
where $U$ is the matrix given in \eqref{def:U}.
Then, there exists $p\to f_p \in C^0(D_\C(3,\varepsilon_0),L^\infty_{1/8})$ satisfying
\begin{align}\label{eq:Sstarf}
    (S^*_p)^2 f_p=2\lambda(p)F_1(p)-L_{-,p}F_2(p)+4\lambda(p)^2a(p)\psi_p,
\end{align}
where $S^*_p =-\phi_p^{-1} \partial_x(\phi_p \cdot)$ and
\begin{align}\label{eq:a(p)}
    a(p)=-\frac{\int_{\R}\phi_p F_1(p)\,dx}{2\lambda(p)\int_{\R} \phi_p \psi_p}.
\end{align}
Moreover, we have
\begin{align}
    f_3(x)=&-\frac{22\sqrt{2}}{15}(x\tanh(x) -\log(e^x+e^{-x}))\cosh(x)\label{eq:f3explicit} \\&+ \frac{2\sqrt{2}}{15}\mathrm{sech}(x) +\frac{14\sqrt{2}}{5}\mathrm{sech}^3(x) -  \frac{8\sqrt{2}}{3}\mathrm{sech}^5(x).\nonumber
\end{align}
\end{lemma}

\begin{proof}
First of all notice that the denominator in \eqref{eq:a(p)} is nonzero by \cite[Proposition 2.9]{W2}.

\noindent By the definition of $S_p^*$, for $\widetilde{F}_p$ is the r.h.s.\ of \eqref{eq:Sstarf}, formally we have
\begin{align}\label{eq:formulafp}
    f_p(x)=\phi_p^{-1}(x) \int_{-\infty}^x \phi_p(y)\cdot
    \phi_p^{-1}(y) \int_0^y \phi_p(z) \widetilde{F}_p(z)\,dz\,dy .
\end{align}
Notice that $\widetilde{F}_p \in L^\infty_{1/4}$, $\int_{\R}\widetilde{F}_p\phi_p=0$ and $\widetilde{F}_p$ is an even function.
We show that \eqref{eq:formulafp} is well defined  and that $f_p\in L^{\infty}_{1/4}$.
First,
\begin{align*}
    \int_0^y \phi_p(z) \widetilde{F}_p(z)\,dz=\int_y^\infty \phi_p(z) \widetilde{F}_p(z)\,dz.
\end{align*}
Thus, for $y>0$, we have
\begin{align*}
\left |\int_0^y \phi_p(z) \widetilde{F}_p(z)\,dz\right |\lesssim \int_y^\infty e^{-z} e^{-z/4}\,dz\lesssim e^{-5|y|/4}.
\end{align*}
For $y<0$, we have the same upper bound.
Therefore, \eqref{eq:formulafp} is well defined. By $\phi_p (x) \gtrsim e ^{-|x|} $ we conclude $f_p\in L^\infty_{1/4}$.
Further, using the continuity of $\phi_p$ and $\widetilde{F}_p$ in $L^\infty_{s}$ ($s<1$) and $L^\infty_{1/4}$, we can show the continuity in $L^\infty_{1/8}$.
Formula \eqref{eq:f3explicit} follows from direct computation.
\end{proof}

We set
\begin{align*}
    L_{2,p}&:=-\partial_x^2+1-\frac{3-p}{p+1}\phi_p^{p-1},\\
    L_{3,p}&:=-\partial_x^2+1-\frac{(3-p)(2-p)}{p+1}\phi_p^{p-1}.
\end{align*}
The following formula is very useful, see \cite{CM24D1} and \cite{Martel2},
\begin{align}\label{eq:intertwine}
    L_{-,p}L_{+,p}(S_p^*)^2=(S_p^*)^2 L_{3,p} L_{2,p}.
\end{align}

\begin{lemma}\label{lem:zeta12}
Let $s>0$.
Then, there exists  $\varepsilon_0>0$ such that for $p\in D_\C(3,\varepsilon_0)$ with $\Im \lambda(p)>0$, the equation
\begin{align}\label{eq:tildezeta}
    \(L_{3,p}L_{2,p}-4\lambda(p)^2\)\widetilde{\zeta}_p=f_p,
\end{align}
has a unique solution in $H^4$ and we have
\begin{align}\label{eq:tildezeta2}
    \begin{pmatrix}
\xi_{(2,0)}(p)\\
\xi_{(0,2)}(p)
    \end{pmatrix}
    =
   U^{-1} \begin{pmatrix}
   \zeta_1(p)\\
   \im \zeta_2(p)
   \end{pmatrix}:=U^{-1}\begin{pmatrix}
(S_p^*)^2\widetilde{\zeta}_p+a(p)\psi_p\\
\frac{\im }{2\lambda(p)}\(L_{+,p}\((S_p^*)^2\widetilde{\zeta}_p+a(p)\psi_p\)+F_2(p)\)
    \end{pmatrix}.
\end{align}
Furthermore, taking $\varepsilon_0>0$ sufficiently small, we have
\begin{align}\label{eq:unifboundtildezeta}
    \sup_{p\in D_\C(3,\varepsilon_0), \Im \lambda(p)>0}\|\widetilde{\zeta}_p\|_{H^4_{-s}}<\infty.
\end{align}
\end{lemma}

\begin{proof}
We first show that if $\widetilde{\zeta}_p$ is given by \eqref{eq:tildezeta}, then the right hand side of \eqref{eq:tildezeta2} solves \eqref{def:xi20}.
Since $\Im \lambda(p)>0$, \eqref{eq:tildezeta2} has a unique solution: hence the right  hand side of \eqref{eq:tildezeta2} is $(\xi_{(2,0)}(p)\ \xi_{(0,2)}(p))^\intercal$.

\noindent Applying $\im U$ to \eqref{def:xi20}, to show that that the right hand side of \eqref{eq:tildezeta2} solves \eqref{def:xi20},  it suffices to show  the following
\begin{align*}
   \(\mathcal{ L}_p-2\im \lambda(p)\)\begin{pmatrix}
   \zeta_1(p)\\
   \im \zeta_2(p)
   \end{pmatrix}
   =
   \begin{pmatrix}
   \im F_1(p)\\
   F_2(p)
   \end{pmatrix},
\end{align*}
which is equivalent to
\begin{align}
    \(L_{-,p}L_{+,p} - 4\lambda(p)^2\)\zeta_1(p) &= 2\lambda(p)F_1(p)-L_{-,p}F_2(p) \text{ and}\label{eq:zeta1}\\
    \zeta_2(p)&=\frac{1}{2\lambda(p)}\(L_{+,p}\zeta_1(p)+F_2(p)\).\label{eq:zeta2}
\end{align}
Applying $(S_p^*)^2$ to \eqref{eq:tildezeta}, by \eqref{eq:intertwine} and \eqref{eq:Sstarf} we obtain
\begin{align*}
(L_{-,p}L_{+,p}-4\lambda(p)^2)(S_p^*)^2\widetilde{\zeta}_p = 2\lambda(p)F_1(p)-L_{-,p}F_2(p)+4\lambda(p)a(p)\psi_p.
\end{align*}
Thus, by $L_{-,p}L_{+,p}\psi(p)=L_{-,p}\phi_p=0$, we have
\begin{align*}
    (L_{-,p}L_{+,p}-4\lambda(p)^2)\((S_p^*)^2\widetilde{\zeta}_p + a(p)\psi(p)\) = 2\lambda(p)F_1(p)-L_{-,p}F_2(p).
\end{align*}
Therefore,   $\zeta_1(p):=(S_p^*)^2\widetilde{\zeta}_p + a(p)\psi(p)$   satisfies \eqref{eq:zeta1}.
Substituting this $\zeta_1$ into \eqref{eq:zeta2}, we get the $\zeta_2$ of \eqref{eq:tildezeta2}.

\noindent We next show that \eqref{eq:tildezeta} can be solved.
We can rewrite \eqref{eq:tildezeta} as
\begin{align*}
    \( (-\partial_x^2+1+2\lambda(p)) (-\partial_x^2+1-2\lambda(p)) -(p-3)B_p \)\widetilde{\zeta}_p =f_p,
\end{align*}
where \small
\begin{align*}
    B_p:=(-\partial_x^2+1)V_{2,p} + V_{3,p}(-\partial_x^2+1) -(p-3)V_{3,p}V_{2,p} \text{ for } V_{2,p}:=\frac{1}{p+1}\phi_p^{p-1}   \text{ and }  V_{3,p}=\frac{(2-p)}{p+1}\phi_p^{p-1}.
\end{align*}\normalsize
Thus, $\widetilde{\zeta}_p$ is given as a fixed point of
\begin{align*}
    \Phi_p[\zeta]:=(-\partial_x^2+1-2\lambda(p))^{-1}(-\partial_x^2+1+2\lambda(p))^{-1}\(f(p)+(p-3)B_p \zeta\).
\end{align*}
For $s>0$ sufficiently small  it is easy to show that $(-\partial_x^2+1+2\lambda(p))^{-1}B_p\cdot \in \mathcal{L}\( L^2_{-s} ,   L^2_{ s} \)$  and $(-\partial_x^2-1-2\lambda(p))^{-1}\in \mathcal{L}\( L^2_{ s} ,   L^2_{ s} \)$ and that for  $p-3$ sufficiently small  $\Phi_p$ is a contraction mapping in $L^2_{-s}$.
Therefore, we have \eqref{eq:tildezeta} and  from elliptic regularity  $\widetilde{\zeta}_p\in H^4$.

\noindent Finally, the uniform bound \eqref{eq:unifboundtildezeta} follows from the uniform bounds of the operators $(-\partial_x^2+1-2\lambda(p))^{-1}:H^n_s\to H^{n+2}_{-s}$, $(-\partial_x^2+1+2\lambda(p))^{-1}:L^2_s\to H^2_s$ and $(-\partial_x^2+1+2\lambda(p))^{-1} B_p:H^n_{-s}\to H^n_s$.
\end{proof}

We now compute the imaginary part of $\widetilde{\zeta}_p$ for $p=3.$
\begin{lemma}\label{lem:extendtildezeta}
For $s>0$ and $\widetilde{\zeta}_p$ the function  in Lemma \ref{lem:zeta12}  we can continuously extend $\widetilde{\zeta}_p$ up to $p=3$ in the topology of $H^4_{-s}$.
\end{lemma}

\begin{proof}
For $p$ near $3$ with $\Im \lambda(p)>0$, we have
\begin{align*}
    (-\partial_x^2+1-2\lambda(p))^{-1}u =   K_p * u \text{ where }
    K_p(x)=\frac{\im e^{\im \sqrt{2\lambda(p)-1}|x|}}{2\sqrt{2\lambda(p)-1}}.
\end{align*}
Thus, setting $(-\partial_x^2-1-\im 0)^{-1}$ by
\begin{align*}
    (-\partial_x^2-3-\im 0)^{-1} u= K_3*u, \ K_3(x):=\frac{\im }{2}e^{\im |x|},
\end{align*}
it is easy to show $\|(-\partial_x^2-1-\im 0)^{-1}\|_{L^2_s\to L^2_{-s}}\lesssim 1$ and
\begin{align}\label{eq:limabcont}
    \lim_{p\to 3, \Im \lambda(p)>0}\| (-\partial_x^2+1-2\lambda(p))^{-1} -(-\partial_x^2-1-\im 0)^{-1}\|_{L^2_s\to L^2_{-s}}=0.
\end{align}
We set
\begin{align}\label{def:tildezeta3}
    \widetilde{\zeta}_3:=(-\partial_x^2-1-\im 0)^{-1}(-\partial_x^2+3)^{-1}f_3.
\end{align}
Then, we have
\begin{align*}
    &\|\widetilde{\zeta}_p-\widetilde{\zeta}_3\|_{L^2_{-s}}=\|\Phi_p[\widetilde{\zeta}_p]-(-\partial_x^2-1-\im 0)^{-1}(-\partial_x^2+3)^{-1}f_3\|_{L^2_{-s}}\\&
    \lesssim \| (-\partial_x^2+1-2\lambda(p))^{-1} -(-\partial_x^2-1-\im 0)^{-1}\|_{L^2_s\to L^2_{-s}} \|(-\partial_x^2+3)^{-1}f_3\|_{L^2_s} \\&\quad + \| f_p-f_3\|_{L^2_s} +|p-3| \| (-\partial_x^3+1+2\lambda(p))B_p \widetilde{\zeta}_p\|_{L^2_s}.
\end{align*}
Thus, by \eqref{eq:limabcont}, the continuity of $f$ in $L^2_s$ (Lemma \ref{lem:fp}), the uniform bound of $\|(-\partial_x^3+1+2\lambda(p))B_p\|_{L^2_{-s}\to L^2_s}$ and \eqref{eq:unifboundtildezeta} give us  the following, and hence the proof of the lemma,
\begin{align*}
    \lim_{p\to 3, \Im \lambda(p)>0} \|\widetilde{\zeta}_p-\widetilde{\zeta}_3\|_{L^2_{-s}}=0.
\end{align*}
\end{proof}

We can now prove Lemma \ref{lem:gamma31}.
\begin{proof}[Proof of Lemma \ref{lem:gamma31}]
From Lemmas \ref{lem:analyg} and \ref{lem:phip-2}, we can analytically extend $g_3$ and $\phi_p^{p-2}$ in $\mathfrak{C}_\delta$ by taking $\delta>0$ sufficiently small.
Thus, with the definition of $\gamma_{3,1}$ and Lemma \ref{lem:xi2002extend}, we see that we can extend $\gamma_{3,1}$ analytically on $\mathfrak{C}_\delta\setminus \mathfrak{D}_\delta$, denoting the extension $\Gamma_1$.
By Lemma \ref{lem:extendtildezeta} and \eqref{eq:tildezeta2}  we can define  $\xi_{(2,0)}$ and $\xi_{(0,2)}$
so that they belong to $C^0 \( \{3\} \cup   \mathfrak{C}_\delta\setminus \mathfrak{D}_\delta  , L^2_{-1/10}\). $
This implies  $\Gamma_1\in C^0 \( \{3\} \cup   \mathfrak{C}_\delta\setminus \mathfrak{D}_\delta   \). $
\end{proof}

It remains to prove Lemma \ref{lem:Gamma13}.

\begin{lemma}\label{lem:tildezeta3}
There exists $b>0$ such that
\begin{align*}
    \Im \widetilde{\zeta}_3 = b\cos(x).
\end{align*}

\end{lemma}

\begin{proof}
We compute $\Im \widetilde{\zeta}_3$ using \eqref{def:tildezeta3}.
By $(-\partial_x^2+3)^{-1}u=(2\sqrt{3})^{-1}e^{-\sqrt{3}|\cdot|}*u$, we have
\begin{align*}
    &\widetilde{\zeta}_3= \frac{\im }{2} e^{\im |\cdot|} * (2\sqrt{3})^{-1}e^{-\sqrt{3}|\cdot|} \\&*\(-\frac{22\sqrt{2}}{15}(x\tanh(x) -\log(e^x+e^{-x}))\cosh(x)  + \frac{2\sqrt{2}}{15}\mathrm{sech}(x) +\frac{14\sqrt{2}}{5}\mathrm{sech}^3(x) -  \frac{8\sqrt{2}}{3}\mathrm{sech}^5(x)\).
\end{align*}
We now compute $e^{\im |x|}* e^{-\sqrt{3}|x|}$.
Since this is even function, it suffices to compute for $x\geq 0$.
\begin{align*}
e^{\im |\cdot|}*e^{-\sqrt{3}|\cdot|}(x)&=\int_{-\infty}^\infty e^{\im|x-y|}e^{-\sqrt{3}|y|}\,dy\\&
=\int_{-\infty}^0e^{\im (x-y)}e^{\sqrt{3}y}\,dy+\int_0^x e^{\im (x-y)}e^{-\sqrt{3}y}\,dy+\int_x^\infty e^{-\im (x-y)}e^{-\sqrt{3}y}\,dy\\&
=
\frac{\sqrt{3}}{2}e^{\im x} + \frac{\im}{2}e^{-\sqrt{3}x}.
\end{align*}
Therefore, we have
\begin{align*}
&\widetilde{\zeta}_3 = \(\im \frac{1}{8}e^{\im |x|}-\frac{1}{8\sqrt{3}}e^{-\sqrt{3}|x|}\)\\&*\(-\frac{22\sqrt{2}}{15}(x\tanh(x) -\log(e^x+e^{-x}))\cosh(x) + \frac{2\sqrt{2}}{15}\mathrm{sech}(x) +\frac{14\sqrt{2}}{5}\mathrm{sech}^3(x) -  \frac{8\sqrt{2}}{3}\mathrm{sech}^5(x)\),
\end{align*}
and in particular
\small
\begin{align*}	\Im \widetilde{\zeta}_3
		=\frac{\sqrt{2}}{60}\cos(x)* \(-11(x\tanh(x) -\log(e^x+e^{-x}))\cosh(x) + \mathrm{sech}(x) +21\mathrm{sech}^3 (x)-  20\mathrm{sech}^5(x)\).
	\end{align*}
	\normalsize
For any even function $g$,
\begin{align*}
	\cos(x)*g(x)&=\int _\R \cos(x-y) g(y)\,dy=\cos(x)\int _\R\cos(y)g(y) +\sin(x)\cancel{\int _\R\sin(y)g(y)}\\&
	= \int_\R \cos(y)g(y) \ dy \    \cos(x).
\end{align*}
Thus, for some real constant $b$ we have
\begin{align*}
    \Im \widetilde{\zeta}= b \cos(x) .
\end{align*}
We now show $b\neq 0$.
Using the notation
\begin{align}
	A_{a,n}&:=\int \sech^n(x) \cos(ax) dx,\label{def:Aan}
\end{align}
we have
\begin{align*}
    b=\frac{\sqrt{2}}{60}\(-11\int \cos (x)(x\tanh(x) -\log (e^x+e^{-x}))\cosh(x) +A_{1,1}+21A_{1,3}-20A_{1,5}\).
\end{align*}
From
\begin{align} 
    		A_{a,n+2}&=\frac{n^2+a^2}{n^2+n} A_{a,n},\label{eq:relAan}
\end{align}
and
\begin{align}\label{Aa1}
    A_{a,1}=\pi\mathrm{sech}\(\frac{\pi}{2}a\),
\end{align}
we have
\begin{align*}
	A_{1,1}+21A_{1,3}-20A_{1,5}
	=(1+21-20\times5/6)A_{1,1}>0.
\end{align*}
Thus, it remains to show
\begin{align}\label{7/25}
	\int \cos(x)(x\tanh(x) -\log(e^x+e^{-x}))\cosh(x)<0.
\end{align}

Now, since the integrand of \eqref{7/25} is even, it suffices to compute the integral from $0$ to $\infty$.
Then,
\begin{align*}
	&\int_0^\infty \cos(x)(x\tanh(x) -\log(e^x+e^{-x}))\cosh(x)\\&=\int_0^\infty \cos(x)\(x\sinh(x)-x\cosh(x)\) -\int_0^\infty \cos(x)\log(1+e^{-2x})\cosh(x).
\end{align*}
By
\begin{align*}
	\cos(x)\(x\sinh(x)-x\cosh(x)\)=-2x\cos x e^{-x} =-2\(\frac{1}{2}e^{-x}\((x+1)\sin (x)-x\cos (x)\)\)',
\end{align*}
we see that the contribution of the first integral is $0$.
Thus, it suffices to show
\begin{align*}
	\int_0^\infty \cos(x)\log(1+e^{-2x})\cosh(x)>0.
\end{align*}
Expanding $\log$ (recall $\log(1+x)=\sum_{n=1}^\infty (-1)^{n-1}n^{-1} x^n$), we have
\begin{align*}
	\int_0^\infty \cos(x)\log(1+e^{-2x})\cosh(x)=\sum_{n=1}^\infty \frac{(-1)^n}{n}\int \cos(s) e^{-2n x}\cosh(x).
\end{align*}
Now, by
\begin{align}\label{7/25/18}
	\int_0^\infty  \cos(x) e^{-2n x}\cosh(x)=\frac{n^3}{4n^4+1}
\end{align}
the problem is reduced to showing
\begin{align*}
	\sum_{n=1}^\infty(-1)^{n-1}\frac{n^2}{4n^4+1}>0.
\end{align*}
This can be shown as follows.
\begin{align*}
	\sum_{n=1}^\infty(-1)^{n-1}\frac{n^2}{4n^4+1}&\geq \frac{1}{5}-\frac{4}{65}+\frac{9}{325}-|\sum_{n=4}^\infty(-1)^{n-1}\frac{n^2}{4n^4+1}|\\&
	\geq \frac{1}{5}-\frac{4}{65}+\frac{9}{325}-\frac{1}{4}\sum_{n=4}^\infty n^{-2}\\&
	=\frac{1}{5}-\frac{4}{65}+\frac{9}{325}+\frac{1}{4}\(1+\frac{1}{4}+\frac{1}{9}\)-\frac{\pi^2}{24}\\&
	=\frac{1}{24}\(\frac{23701}{1950}-\pi^2\)\\&>\frac{1}{24}\(10-\pi^2\)>0,
\end{align*}
where we have used $\sum_{n=1}^\infty n^{-2}=\pi^2/6$.
Therefore, we conclude $b> 0$.
\end{proof}

We compute the imaginary part of $\zeta_1$ and $\zeta_2$ explicitly at $p=3$.

\begin{lemma}\label{claim:zeta}
The functions $\zeta_1$ and $\zeta_2$ given in Lemma \ref{lem:zeta12} can be continuously extended up to $p=3$.
Further, setting $\zeta_1(3)$ and $\zeta_2(3)$ the continuous extension of $\zeta_1, \zeta_2$, we have
\begin{align}
    \Im\zeta_1(3)&
    =2b\(-\sech^2(x)\cos(x) +\tanh(x)\sin(x)\),\label{eq:zeta13}\\
    \Im \zeta_2(3)&
    = b\(2-2\sech^2(x)\)\tanh(x)\sin(x),\label{eq:zeta23}
\end{align}
where $b>0$ is the constant given in Lemma \ref{lem:tildezeta3}.
\end{lemma}

\begin{proof}
Because $a$ and $\psi_p$ are $\R$-valued at $p=3$, from \eqref{eq:tildezeta2}, we have
\begin{align*}
    \Im \zeta_1(3)=\Im (S_3^*)^2\widetilde{\zeta}_3.
\end{align*}
Thus, from Lemma \ref{lem:tildezeta3} and $(S_3^*)^2f=\mathrm{cosh}(x)\(\sech(x) f \)''$, we have \eqref{eq:zeta13}.

Next, by \eqref{eq:tildezeta2}, we have
\begin{align*}
    \zeta_2(3)=\frac{1}{2}\(L_{+,3}\zeta_1(3)+F_2(3)\).
\end{align*}
Therefore, since $F_2(3)$ is $\R$-valued, we have
\begin{align*}
    \Im \zeta_2(3)=\frac{1}{2}L_{+,3} \Im \zeta_1(3).
\end{align*}
Then, by direct computation, we have \eqref{eq:zeta23}.
\end{proof}

We are now in the position to prove Lemma \ref{lem:Gamma13}.

\begin{proof}[Proof of Lemma \ref{lem:Gamma13}]

From the proof of Lemma \ref{lem:gamma31}, we see that $\Gamma_1(3)$, the continuous extension of $\Gamma_1$ defined on $\mathfrak{A}_1$ to $p=3$, is given by
\begin{align}\nonumber
	\Gamma_1(3)=&2\<\phi_3\(\xi_{(1,0)}(3)\xi_{(2,0)}(3)+\xi_{(1,0)}(3)\xi_{(0,2)}(3)+\xi_{(0,1)}(3)\xi_{(2,0)}(3)\),g_{(3,0)}(3)\>\\&+2\<\phi_3\(\xi_{(1,0)}(3)\xi_{(0,2)}(3)+\xi_{(0,1)}(3)\xi_{(2,0)}(3)+\xi_{(0,1)}(3)\xi_{(0,2)}(3)\),g_{(0,3)}(3)\>,\label{eq:Gamma1p3}
\end{align}
In the following, we drop $3$ and just write $\phi$ for $\phi_3$, $\xi_{(1,0)}$ for $\xi_{(1,0)}(3)$, etc.
Now, recall that the  first equality of \eqref{eq:tildezeta2}, we have
\begin{align*}
\begin{pmatrix}
\xi_{(2,0)}\\
\xi_{(0,2)}
\end{pmatrix}
=
U^{-1}
\begin{pmatrix}
\zeta_1\\
\im \zeta_2
\end{pmatrix}
=\frac{1}{2}\begin{pmatrix}
\zeta_1+\zeta_2\\
\zeta_1-\zeta_2
\end{pmatrix}.
\end{align*}
Substituting this into \eqref{eq:Gamma1p3}, we have
\begin{align}
&	\Gamma_1(3)=\<\xi_{(1,0)}(\zeta_1+\zeta_2)+\xi_{(1,0)}(\zeta_1-\zeta_2)+\xi_{(0,1)}(\zeta_1+\zeta_2),\phi g_{(3,0)}\>\nonumber
	\\&\quad+\<\xi_{(1,0)}(\zeta_1-\zeta_2)+\xi_{(0,1)}(\zeta_1+\zeta_2)+\xi_{(0,1)}(\zeta_1-\zeta_2),\phi g_{(0,3)}\>\nonumber\\&
	=\<\zeta_1,\phi(2\xi_{(1,0)}+\xi_{(0,1)})g_{(3,0)}+\phi (\xi_{(1,0)}+2\xi_{(0,1)}) g_{(0,3)}\>+\<\zeta_2,\phi \xi_{(0,1)} g_{(3,0)} -\phi \xi_{(1,0)}g_{(0,3)}\>.\label{eq:Gamma1p32}
\end{align}
By \eqref{eq:analyg1} and \eqref{eq:xiat3}, $\xi_{(1,0)},\xi_{(0,1)},g_{(3,0)}$ and $g_{(0,3)}$ are $\R$-valued.
Thus, $\Im \Gamma_1(3)$ is given by replacing $\zeta_1$ and $\zeta_2$ by $\Im \zeta_1$ and $\Im \zeta_2$ in \eqref{eq:Gamma1p32}.
By \eqref{eq:analyg1} and \eqref{eq:xiat3} we have
\begin{align*}
    &\phi(2\xi_{(1,0)}+\xi_{(0,1)})g_{(3,0)}+\phi (\xi_{(1,0)}+2\xi_{(0,1)}) g_{(0,3)}\\&=
    	\sqrt{2}\(2\mathrm{sech} (x)-6\mathrm{sech}^5(x)\)\cos(\sqrt{2}x)+ \(-8\mathrm{sech}(x)+12\mathrm{sech}^3(x)\)\mathrm{tanh}\sin(\sqrt{2}x),\\
    &\phi \xi_{(0,1)} g_{(3,0)} -\phi \xi_{(1,0)}g_{(0,3)}=-2\sqrt{2} \mathrm{sech}^5(x)\cos(\sqrt{2}x) + 4\mathrm{sech}^3(x)\mathrm{tanh}(x) \sin(\sqrt{2}x).
\end{align*}
Thus, with \eqref{eq:zeta13} and \eqref{eq:zeta23}, we have
\begin{align*}
&\Im \Gamma_1(3)=4b\<-\mathrm{sech}^2(x)\cos(x) +\tanh(x)\sin(x),\sqrt{2}\(\mathrm{sech}(x) -3\mathrm{sech}^5(x)\)\cos(\sqrt{2}x) \>\\&\quad
+4b\<-\mathrm{sech}^2(x)\cos(x) +\tanh(x)\sin(x), \(-4\mathrm{sech}(x)+6\mathrm{sech}^3(x)\)\tanh(x)\sin(\sqrt{2}x)\>\\&\quad
+4b\<\(1-\mathrm{sech}^2(x)\)\tanh(x)\sin(x),-\sqrt{2} \mathrm{sech}^5(x)\cos(\sqrt{2}x) + 2\mathrm{sech}^3(x)\tanh(x) \sin(\sqrt{2}x)\>.
\end{align*}
Using the notation $A_{a,n}$ given in \eqref{def:Aan},
\begin{align*}
    B_{a,n}:=\int \sech^n(x) \tanh(x) \sin(ax),
\end{align*}
and elementary formula for trigonometric functions such as $$\cos(x) \cos(\sqrt{2}x)= \frac{1}{2}\(\cos((\sqrt{2}-1)x)+\cos((\sqrt{2}+1)x)\),$$
we can express $\Im \Gamma_1(3)$ as
%
\begin{align*}
\frac{1}{8b}\Im\Gamma_1(3)=&
-4A_{\sqrt{2}-1,1}+(12-\sqrt{2})A_{\sqrt{2}-1,3}-10A_{\sqrt{2}-1,5}+(3\sqrt{2}+2)A_{\sqrt{2}-1,7}\\&
+4A_{\sqrt{2}+1,1}+(-12-\sqrt{2})A_{\sqrt{2}+1,3}+10A_{\sqrt{2}+1,5}+(3\sqrt{2}-2)A_{\sqrt{2}+1,7}\\&
-\sqrt{2} B_{\sqrt{2}-1,1}+4B_{\sqrt{2}-1,3}+(-6+4\sqrt{2})B_{\sqrt{2}-1,5}-\sqrt{2} B_{\sqrt{2}-1,7}\\&
+\sqrt{2} B_{\sqrt{2}+1,1}+4B_{\sqrt{2}+1,3}+(-6-4\sqrt{2})B_{\sqrt{2}+1,5}+\sqrt{2} B_{\sqrt{2}+1,7}.
\end{align*}
Using the formula \eqref{Aa1}, \eqref{eq:relAan} and
\begin{align*}
    B_{a,n}=\frac{a}{n}A_{a,n},
\end{align*}
we arrive at
\begin{align*}
\frac{1}{8b}\Im \Gamma_1(3)&=\( -\frac{164}{105}+ \frac{59}{63}\sqrt{2} \)\pi \mathrm{sech}(\frac{\pi}{2}(\sqrt{2}-1))+\( \frac{164}{105}+ \frac{59}{63}\sqrt{2}\)\pi \mathrm{sech}(\frac{\pi}{2}(\sqrt{2}+1))\\&\approx
-0.203\neq 0.
\end{align*}
Therefore, we have the conclusion.
\end{proof}

\section*{Acknowledgments}
We are grateful for the anonymous referee for pointing the result Proposition 1.2 (4).
C. was supported   by the Prin 2020 project \textit{Hamiltonian and Dispersive PDEs} N. 2020XB3EFL.
M.  was supported by the JSPS KAKENHI Grant Number 19K03579, 23H01079 and 24K06792.

Department of Mathematics and Geosciences,  University
of Trieste, via Valerio  12/1  Trieste, 34127  Italy.
{\it E-mail Address}: {\tt scuccagna@units.it}

Department of Mathematics and Informatics,
Graduate School of Science,
Chiba University,
Chiba 263-8522, Japan.
{\it E-mail Address}: {\tt maeda@math.s.chiba-u.ac.jp}

\end{document}